\documentclass[11pt]{article}
\usepackage{geometry}                
\usepackage{amsmath,amssymb}
\usepackage{amsthm}
\usepackage{hyperref}
\usepackage{graphicx,epsfig,color}
\usepackage{booktabs,bm,multirow}
\usepackage{cases}
\usepackage[caption=false]{subfig}

\newtheorem{theorem}{Theorem}

\newtheorem{remark}[theorem]{Remark}
\newtheorem{lemma}[theorem]{Lemma}

\input{PFRIM.sty}

\title{Pseudoinverse-free randomized block iterative algorithms for consistent and inconsistent linear systems}
\author{Kui Du\thanks{School of Mathematical Sciences and Fujian Provincial Key Laboratory of Mathematical Modeling and High Performance Scientific Computing, Xiamen University, Xiamen 361005, China (kuidu@xmu.edu.cn).},\quad  Xiao-Hui Sun\thanks{School of Mathematical Sciences, Xiamen University, Xiamen 361005, China ({sunxh@stu.xmu.edu.cn}).}} 
\date{}                                           

\begin{document}
\maketitle

\begin{abstract} Randomized iterative algorithms have attracted much attention in recent years because they can approximately solve large-scale linear systems of equations without accessing the entire coefficient matrix. In this paper, we propose two novel pseudoinverse-free randomized block iterative algorithms for solving consistent and inconsistent linear systems. The proposed algorithms require two user-defined random matrices: one for row sampling and the other for column sampling. We can recover the well-known doubly stochastic Gauss--Seidel, randomized Kaczmarz, randomized coordinate descent, and randomized extended Kaczmarz algorithms by choosing appropriate random matrices used in our algorithms. Because our algorithms allow for a much wider selection of these two random matrices, a number of new specific algorithms can be obtained. We prove the linear convergence in the mean square sense of our algorithms. Numerical experiments for linear systems with synthetic and real-world coefficient matrices demonstrate the efficiency of some special cases of our algorithms. 
\vspace{.5mm} 



\end{abstract}

\section{Introduction} 
 
Row-action iterative algorithms such as the Kaczmarz algorithm \cite{Kaczmarz1937angen} (also known in computerized tomography as the algebraic reconstruction technique \cite{natterer2001mathe}) are widely used to solve a linear system of equations \beq\label{clin}{\bf Ax=b},\quad \mbf A\in\mbbr^{m\times n}, \quad \mbf b\in\mbbr^m.\eeq They do not need to compute {\it entire} matrix-vector multiplications, and each iteration only requires a sample of rows of the coefficient matrix. Numerical experiments show that using the rows of the coefficient matrix in random order rather than in their given order can often greatly improve the convergence \cite{herman1993algeb,natterer2001mathe}. In a seminal paper \cite{strohmer2009rando}, Strohmer and Vershynin proposed a randomized Kaczmarz algorithm which converges linearly in the mean square sense to a solution of a consistent linear system. The convergence rate of the randomized Kaczmarz algorithm depends only on the scaled condition number of the coefficient matrix. Many subsequent studies on the development and analysis of randomized iterative algorithms for consistent and inconsistent linear systems of equations have been triggered; see, for example, \cite{leventhal2010rando,needell2010rando,zouzias2013rando,dumitrescu2015relat,gower2015rando,ma2015conve,ma2018itera,bai2019parti,du2019tight,razaviyayn2019linea,guan2020note}. Variants based on a variety of acceleration strategies have also been proposed; see, for example, \cite{liu2016accel,bai2018greed,bai2018relax,bai2019greed,liu2019varia,necoara2019faste,du2020rando,richtarik2020stoch,zhang2020relax,moorman2021rando,rebrova2021block,liu2021greed,wu2021two}.

In this paper, we propose a {\it doubly stochastic block iterative algorithm} and an {\it extended block row sampling iterative algorithm} for solving {\it consistent and inconsistent} linear systems. Our algorithms employ two user-defined {\it discrete or continuous} random matrices: one for row sampling and the other for column sampling. By choosing appropriate random matrices in our algorithms, we recover the doubly stochastic Gauss--Seidel (DSGS) algorithm \cite{razaviyayn2019linea}, the randomized Kaczmarz (RK) algorithm \cite{strohmer2009rando}, the randomized coordinate descent (RCD) algorithm \cite{leventhal2010rando}, and the randomized extended Kaczmarz (REK) algorithm \cite{zouzias2013rando}. We emphasize that our algorithms are pseudoinverse-free and therefore different from projection-based block algorithms (which need to solve a small least-squares problem or, equivalently, apply a pseudoinverse to a vector at each iteration), for example, those in  \cite{needell2014paved,needell2015rando,gower2015rando,liu2021greed,wu2021two}. We prove the linear convergence in the mean square sense of our algorithms. As the convergence results hold for a wide range of distributions, more efficient block cases of the general algorithms can be designed. Numerical results are reported to illustrate the efficiency of some special cases of our algorithms.

{\it Main theoretical results}. For arbitrary initial guess $\mbf x^0\in\mbbr^n$, we define the vector $$\mbf x_\star^0:=\mbf A^\dag\mbf b+(\mbf I-\mbf A^\dag\mbf A)\mbf x^0,$$ which is a solution if $\bf Ax=b$ is consistent, or a least squares solution if $\bf Ax=b$ is inconsistent. We mention that $\mbf x_\star^0$ is the orthogonal projection of $\mbf x^0$ onto the set $$\{\mbf x\in\mbbr^n \ |\  \bf A^\top Ax = A^\top b\}.$$ The main theoretical results of this work are as follows. 

\bit
\item[(1)] The block {\it row} sampling iterative algorithm (one special case of the doubly stochastic block iterative algorithm; see Section 2.1) converges linearly in the mean square sense to $\mbf x_\star^0$ if $\bf Ax=b$ is consistent (Theorem \ref{rslt}) and to within a radius of $\mbf x_\star^0$ if $\bf Ax=b$ is inconsistent (Theorem \ref{rslti}). 
\item[(2)] The block {\it column} sampling iterative algorithm (another special case of the doubly stochastic block iterative algorithm; see Section 2.2) converges linearly in the mean square sense to $\mbf A^\dag\mbf b$ if $\mbf A$ has full column rank (Theorem \ref{cslt}).
\item[(3)] The extended block row sampling iterative algorithm converges linearly in the mean square sense to $\mbf x_\star^0$ for arbitrary linear system $\bf Ax=b$ (we make no assumptions about the dimensions or rank of the coefficient matrix $\mbf A$ and the system can be consistent or inconsistent; see Theorem \ref{sel}). 
\eit

{\it Organization of the paper}. In Section 2 we propose the doubly stochastic block iterative algorithm (including its special cases) and construct the convergence theory. In Section 3 we propose the extended block row sampling algorithm and prove its linear convergence for arbitrary linear systems. We report the numerical results in Section 4. Finally, we present brief concluding remarks in Section 5.

{\it Notation}. For any random variable $\bm\xi$, we use $\mbbe\bem\bm\xi\eem$ to denote the expectation of $\bm\xi$. For an integer $m\geq 1$, let $[m]:=\{1,2,3,\ldots,m\}$. For any vector $\mbf b\in\mbbr^m$, we use $\mbf b_i$, $\bf b^\top$ and $\|\mbf b\|$ to denote the $i$th entry, the transpose and the Euclidean norm of $\mbf b$, respectively. We use $\mbf I$ to denote the identity matrix whose order is clear from the context. For any matrix $\mbf A\in\mbbr^{m\times n}$, we use $\mbf A_{i,j}$, $\mbf A_{i,:}$, $\mbf A_{:,j}$, $\mbf A^\top$, $\mbf A^\dag$, $\|\mbf A\|$, $\|\mbf A\|_\rmf$, $\ran(\mbf A)$, $\rank(\mbf A)$, $\sigma_{\rm max}(\mbf A)$ and $\sigma_{\rm min}(\mbf A)$ to denote the $(i,j)$ entry, the $i$th row, the $j$th column, the transpose, the Moore--Penrose pseudoinverse, the spectral norm, the Frobenius norm, the column space, the rank, the maximum and the minimum nonzero singular values of $\mbf A$, respectively. If $\rank(\mbf A)=r$, we also denote all the nonzero singular values of $\mbf A$ by $\sigma_1(\mbf A)\geq \sigma_2(\mbf A)\geq\cdots\geq\sigma_r(\mbf A)>0$. For index sets $\mcali\subseteq[m]$ and $\mcalj\subseteq[n]$, let $\mbf A_{\mcali,:}$, $\mbf A_{:,\mcalj}$, and $\mbf A_{\mcali,\mcalj}$ denote the row submatrix indexed by $\mcali$, the column submatrix indexed by $\mcalj$, and the submatrix that lies in the rows indexed by $\mcali$ and the columns indexed by $\mcalj$, respectively. We use $|\mcali|$ to denote the cardinality of a set $\mcali \subseteq [m]$. Given a symmetric matrix $\mbf A$, we use $\lambda_{\max}(\mbf A)$ to denote the largest eigenvalue of $\mbf A$. Given two symmetric matrices $\mbf A$ and $\mbf B$, we use $\bf A\succeq B$ to denote that $\bf A-B$ is positive semidefinite.
 
{\it Preliminary}. The following lemma will be used, and its proof is straightforward.
\begin{lemma}\label{leqd} Let $\alpha>0$, $\beta>0$, and $\mbf A\in\mbbr^{m\times n}$ be any nonzero matrix with $\rank(\mbf A)=r$. For all $\mbf u\in\ran(\mbf A^\top)$, and $0\leq i\leq k$, it holds $$\|(\mbf I-\beta{\bf A^\top A})^i(\mbf I-\alpha{\bf A^\top A})^{k-i}\mbf u\|\leq\delta^k\|\mbf u\|,$$  where $$\delta=\max_{1\leq i\leq r}\{|1-\alpha\sigma_i^2(\mbf A)|,|1-\beta\sigma_i^2(\mbf A)|\}.$$
\end{lemma}

\section{The doubly stochastic block iterative algorithm}\label{ds}

Given an arbitrary initial guess $\mbf x^0\in\mbbr^n$, the $k$th iterate of the doubly stochastic block iterative (DSBI) algorithm is defined as\beq\label{dsbi}\mbf x^k=\mbf x^{k-1}-\alpha\mbf T\mbf T^\top\mbf A^\top\mbf S\mbf S^\top(\mbf A\mbf x^{k-1}-\mbf b),\eeq where the stepsize parameter $\alpha>0$, and the random parameter matrix pair $(\mbf S,\mbf T)$ is sampled independently in each iteration from a distribution $\mcald$ and satisfies $$\mbbe\bem \mbf T\mbf T^\top\mbf A^\top\mbf S\mbf S^\top \eem=\mbf A^\top.$$ We note that the random parameter matrices $\mbf S\in\mbbr^{m\times p}$ and $\mbf T\in\mbbr^{n\times q}$ can be independent or not. We also emphasize that  we do not restrict the numbers of columns of $\mbf S$ and $\mbf T$; indeed, we allow $p$ and $q$ to vary (and hence $p$ and $q$ are random variables). Let $\mbbe_{k-1}\bem\cdot\eem$ denote the conditional expectation conditioned on the first $k-1$ iterations of the DSBI algorithm. We have $$\mbbe_{k-1}\bem\mbf x^k \eem = \mbf x^{k-1}-\alpha\mbf A^\top(\mbf A\mbf x^{k-1}-\mbf b),$$ which is the update of the Landweber iteration \cite{landweber1951itera}. We note that the DSGS algorithm \cite{razaviyayn2019linea} is one special case of the DSBI algorithm. Let the index pair $(i,j)$ be randomly selected with probability $\dsp\frac{|\mbf A_{i,j}|^2}{\|\mbf A\|_\rmf^2}.$  Setting $\dsp\mbf S=\frac{\|\mbf A\|_\rmf}{|\mbf A_{i,j}|}\mbf I_{:,i}$ and $\mbf T=\mbf I_{:,j}$ in (\ref{dsbi}), we have \begin{align*}\mbbe\bem\mbf T\mbf T^\top\mbf A^\top\mbf S\mbf S^\top\eem &=\|\mbf A\|_\rmf^2\mbbe\bem\dsp\frac{\mbf I_{:,j}(\mbf I_{:,j})^\top\mbf A^\top\mbf I_{:,i}(\mbf I_{:,i})^\top}{|\mbf A_{i,j}|^2}\eem\\ &=\|\mbf A\|_\rmf^2\sum_{i=1}^m\sum_{j=1}^n\frac{\mbf I_{:,j}(\mbf I_{:,j})^\top\mbf A^\top\mbf I_{:,i}(\mbf I_{:,i})^\top}{|\mbf A_{i,j}|^2}\frac{|\mbf A_{i,j}|^2}{\|\mbf A\|_\rmf^2}=\mbf A^\top,\end{align*} and $$\mbf x^k=\mbf x^{k-1}-\alpha \frac{\|\mbf A\|_\rmf^2}{\mbf A_{i,j}} \mbf I_{:,j}(\mbf A_{i,:}\mbf x^{k-1}-\mbf b_i),$$ which is the $k$th iterate of the DSGS algorithm.

In the following, we shall present two convergence results (Theorems \ref{eiters} and \ref{general}) of the DSBI algorithm: the first is the convergence of $\|\mbbe\bem \mbf x^k\eem-\mbf x_\star^0\|$ for arbitrary linear systems, and the second is the convergence of $\mbbe\bem\|\mbf x^k-\mbf A^\dag \mbf b\|^2\eem$ for linear systems with full column rank $\mbf A$.

\begin{theorem}\label{eiters} For arbitrary $\mbf x^0\in\mbbr^n$, the $k$th iterate $\mbf x^k$ of the {\rm DSBI} algorithm satisfies \beq\label{ef}\mbbe\bem \mbf x^k-\mbf x_\star^0\eem=(\mbf I-\alpha\mbf A^\top\mbf A)^k(\mbf x^0-\mbf x_\star^0).\eeq Moreover, \beq\label{nef}\|\mbbe\bem \mbf x^k\eem-\mbf x_\star^0\|\leq\l(\max_{1\leq i\leq r}\l|1-\alpha\sigma_i^2(\mbf A)\r|\r)^k\|\mbf x^0-\mbf x_\star^0\|.\eeq
\end{theorem}
\begin{proof} By $\mbf A^\top\mbf A\mbf x_\star^0=\mbf A^\top \mbf b$ and straightforward calculations, we have 
\begin{align*}
\mbbe_{k-1}\bem\mbf x^k-\mbf x_\star^0\eem 
&=\mbf x^{k-1}-\mbf x_\star^0-\alpha\mbbe\bem\mbf T\mbf T^\top\mbf A^\top\mbf S\mbf S^\top\eem(\mbf A\mbf x^{k-1}-\mbf b)\\ 
&=\mbf x^{k-1}-\mbf x_\star^0-\alpha\mbf A^\top(\mbf A\mbf x^{k-1}-\mbf b) \\ 
&=\mbf x^{k-1}-\mbf x_\star^0-\alpha\mbf A^\top\mbf A(\mbf x^{k-1}-\mbf x_\star^0)\\ 
&=(\mbf I-\alpha\mbf A^\top\mbf A)(\mbf x^{k-1}-\mbf x_\star^0).\end{align*}
Then, by the law of total expectation, we have $$\mbbe\bem\mbf x^k-\mbf x_\star^0\eem=(\mbf I-\alpha\mbf A^\top\mbf A)\mbbe\bem\mbf x^{k-1}-\mbf x_\star^0\eem.$$ Unrolling the recurrence yields the formula (\ref{ef}). By $\mbf x^0-\mbf x_\star^0=\mbf A^\dag\mbf A\mbf x^0-\mbf A^\dag\mbf b\in\ran(\mbf A^\top)$ and Lemma \ref{leqd}, we obtain the estimate (\ref{nef}).
\end{proof}
\begin{remark}
	In Theorem \ref{eiters}, no assumptions about the dimensions or rank of $\mbf A$ are assumed, and the system $\bf Ax=b$ can be consistent or inconsistent. If $0<\alpha<2/\sigma_{\max}^2(\mbf A)$, then $\dsp\max_{1\leq i\leq r}\l|1-\alpha\sigma_i^2(\mbf A)\r|<1$. This means $\mbf x^k$ is an asymptotically unbiased estimator for $\mbf x_\star^0$.
\end{remark}

\begin{theorem}\label{general} Let $\mbf A$ have full column rank. For arbitrary $\mbf x^0\in\mbbr^n$ and $\ve>0$, if $0<\alpha<\dsp\frac{2\sigma_{\min}^2(\mbf A)}{(1+\ve)\beta}$, then the $k$th iterate $\mbf x^k$ of the {\rm DSBI} algorithm satisfies $$\mbbe\bem\|\mbf x^k-\mbf A^\dag\mbf b\|^2\eem\leq\eta^k\|\mbf x^0-\mbf A^\dag\mbf b\|^2+\frac{\alpha(1+1/\ve)\gamma(1-\eta^k)}{2\sigma_{\min}^2(\mbf A)-(1+\ve)\alpha\beta},$$ where $$\eta=1-2\alpha\sigma_{\min}^2(\mbf A)+(1+\ve)\alpha^2\beta, \quad  \quad \beta=\|\mbbe\bem\mbf A^\top \mbf S \mbf S^\top\mbf A \mbf  T\mbf T^\top \mbf T\mbf T^\top\mbf A^\top\mbf S\mbf S^\top\mbf A\eem\|,$$ and $$ \gamma=\mbbe\bem\|\mbf T\mbf T^\top\mbf A^\top\mbf S\mbf S^\top(\mbf A\mbf A^\dag\mbf b-\mbf b)\|^2\eem.$$ 
\end{theorem}
\begin{proof} Straightforward calculations yield 
\begin{align}
\|\mbf x^k-\mbf A^\dag\mbf b\|^2 
& =\|\mbf x^{k-1}-\mbf A^\dag\mbf b\|^2-2\alpha(\mbf x^{k-1}-\mbf A^\dag\mbf b)^\top	\mbf T\mbf T^\top\mbf A^\top\mbf S\mbf S^\top(\mbf A\mbf x^{k-1}-\mbf b)\nn \\ 
&\quad +\alpha^2\|\mbf T\mbf T^\top\mbf A^\top\mbf S\mbf S^\top(\mbf A\mbf x^{k-1}-\mbf b)\|^2. \label{nsl}\end{align}
By (\ref{nsl}), $\mbbe\bem \mbf T\mbf T^\top\mbf A^\top\mbf S\mbf S^\top \eem=\mbf A^\top$, and $\mbf A^\top\mbf A\mbf A^\dag\mbf b=\mbf A^\top\mbf b$, we have \begin{align}
\mbbe_{k-1}\bem\|\mbf x^k-\mbf A^\dag\mbf b\|^2\eem
& =\|\mbf x^{k-1}-\mbf A^\dag\mbf b\|^2-2\alpha(\mbf x^{k-1}-\mbf A^\dag\mbf b)^\top	 \mbf A^\top \mbf A(\mbf x^{k-1}-\mbf A^\dag\mbf b)\nn \\ 
&\quad +\alpha^2\mbbe_{k-1}\bem\|\mbf T\mbf T^\top\mbf A^\top\mbf S\mbf S^\top(\mbf A\mbf x^{k-1}-\mbf b)\|^2\eem.\label{km1}
\end{align} It follows from $\mbf A$ has full column rank that
\beq\label{nm1}(\mbf x^{k-1}-\mbf A^\dag\mbf b)^\top\mbf A^\top \mbf A(\mbf x^{k-1}-\mbf A^\dag\mbf b)\geq\sigma_{\min}^2(\mbf A)\|\mbf x^{k-1}-\mbf A^\dag\mbf b\|^2. \eeq
By triangular inequality and Young's inequality, we have
\begin{align} &\ \quad \|\mbf T\mbf T^\top\mbf A^\top\mbf S\mbf S^\top(\mbf A\mbf x^{k-1}-\mbf b)\|^2= \|\mbf T\mbf T^\top\mbf A^\top\mbf S\mbf S^\top(\mbf A\mbf x^{k-1}-\mbf A\mbf A^\dag\mbf b+\mbf A\mbf A^\dag\mbf b-\mbf b)\|^2\nn \\
&\leq (\|\mbf T\mbf T^\top\mbf A^\top\mbf S\mbf S^\top(\mbf A\mbf x^{k-1}-\mbf A\mbf A^\dag\mbf b)\|+\|\mbf T\mbf T^\top\mbf A^\top\mbf S\mbf S^\top(\mbf A\mbf A^\dag\mbf b-\mbf b)\|)^2\nn \\
&\leq (1+\ve)\|\mbf T\mbf T^\top\mbf A^\top\mbf S\mbf S^\top\mbf A(\mbf x^{k-1}-\mbf A^\dag\mbf b)\|^2+ (1+ 1/\ve )\|\mbf T\mbf T^\top\mbf A^\top\mbf S\mbf S^\top(\mbf A\mbf A^\dag\mbf b-\mbf b)\|^2.\label{last1}
\end{align}  Note that \begin{align}&\quad\ \mbbe_{k-1}\bem\|\mbf T\mbf T^\top\mbf A^\top\mbf S\mbf S^\top\mbf A(\mbf x^{k-1}-\mbf A^\dag\mbf b)\|^2\eem\nn \\ &=(\mbf x^{k-1}-\mbf A^\dag\mbf b)^\top\mbbe\bem\mbf A^\top \mbf S \mbf S^\top\mbf A \mbf  T\mbf T^\top \mbf T\mbf T^\top\mbf A^\top\mbf S\mbf S^\top\mbf A\eem(\mbf x^{k-1}-\mbf A^\dag\mbf b)\nn \\ &\leq \beta\|\mbf x^{k-1}-\mbf A^\dag\mbf b\|^2.\label{beta}\end{align}
Combining (\ref{km1}), (\ref{nm1}), (\ref{last1}), and (\ref{beta}) yields
\begin{align*}\mbbe_{k-1}\bem\|\mbf x^k-\mbf A^\dag\mbf b\|^2\eem & \leq \|\mbf x^{k-1}-\mbf A^\dag\mbf b\|^2-2\alpha\sigma_{\min}^2(\mbf A)\|\mbf x^{k-1}-\mbf A^\dag\mbf b\|^2\\
&\quad +(1+\ve)\alpha^2\beta\|\mbf x^{k-1}-\mbf A^\dag\mbf b\|^2+ (1+1/\ve)\alpha^2\gamma\\
&= \eta\|\mbf x^{k-1}-\mbf A^\dag\mbf b\|^2 +(1+1/\ve)\alpha^2\gamma. 
\end{align*}
Then the expected squared norm of the error can be bounded by 
\begin{align*}
\mbbe\bem\|\mbf x^k-\mbf A^\dag\mbf b\|^2\eem &\leq \eta\mbbe\bem\|\mbf x^{k-1}-\mbf A^\dag\mbf b\|^2\eem +(1+1/\ve)\alpha^2\gamma \\
&\leq \eta^k\|\mbf x^0-\mbf A^\dag\mbf b\|^2+(1+1/\ve)\alpha^2\gamma\sum_{i=0}^{k-1}\eta^i\\ 
&= \eta^k\|\mbf x^0-\mbf A^\dag\mbf b\|^2+\frac{\alpha^2(1+1/\ve)\gamma(1-\eta^k)}{1-\eta}\\ 
&=\eta^k\|\mbf x^0-\mbf A^\dag\mbf b\|^2+\frac{\alpha(1+1/\ve)\gamma(1-\eta^k)}{2\sigma_{\min}^2(\mbf A)-(1+\ve)\alpha\beta}.
\end{align*} This completes the proof. 
\end{proof}

Theorem \ref{general} means that, for a full column rank consistent linear system (for which we have $\gamma=0$), the {\rm DSBI} algorithm with sufficiently small $\alpha$ converges linearly to the unique solution in the mean square sense. But, on the other hand, a small parameter $\alpha$ straightforwardly implies very slow convergence $(\eta\approx 1)$.  If more about the random parameter matrices $\mbf S$ and $\mbf T$ are available, then improved convergence results can be obtained.  In the following subsections, we discuss the case $\mbf T=\mbf I$ and the case $\mbf S=\mbf I$, respectively.

\subsection{Block row sampling} In this subsection, we consider the case $\mbf T=\mbf I$ and refer to  the resulting algorithm as the block row sampling iterative (BRSI) algorithm. Given an arbitrary initial guess $\mbf x^0\in\mbbr^n$, the $k$th iterate of the BRSI algorithm is  \beq\label{brs}\mbf x^k=\mbf x^{k-1}-\alpha_{\rm r}\mbf A^\top\mbf S\mbf S^\top(\mbf A\mbf x^{k-1}-\mbf b),\eeq where the stepsize parameter $\alpha_{\rm r}>0$, and the random parameter matrix $\mbf S$ is sampled independently in each iteration from a distribution $\mcald_{\rm r}$ and satisfies $\mbbe\bem \mbf S\mbf S^\top \eem=\mbf I.$ Various choices for $\mbf S$ can be used, e.g., see \cite{chung2017stoch}. 
 
In the following, we shall present the convergence of $\mbbe\bem\|\mbf x^k-\mbf x_\star^0\|^2\eem$ for the consistent case and the inconsistent case in Theorems \ref{rslt} and \ref{rslti}, respectively. For the consistent case, Theorem \ref{rslt} shows that the BRSI algorithm converges linearly to a solution. For the inconsistent case, Theorem \ref{rslti} shows that the BRSI algorithm can only converge to within a radius (convergence horizon) of a least squares solution. Throughout, we define $$\lambda_{\max}^{\rm r}=\max_{\mbf S\sim\mcald_{\rm r}}\lambda_{\max}(\mbf A^\top \mbf S\mbf S^\top\mbf A).$$

\begin{theorem}\label{rslt} Assume that $0<\alpha_{\rm r}<2/\lambda_{\max}^{\rm r}$. If $\bf Ax=b$ is consistent, then for arbitrary $\mbf x^0\in\mbbr^n$, the $k$th iterate $\mbf x^k$ of the {\rm BRSI} algorithm satisfies $$\mbbe\bem \|\mbf x^k-\mbf x_\star^0\|^2\eem\leq \eta_{\rm r}^k\|\mbf x^0-\mbf x_\star^0\|^2,$$ where $$\eta_{\rm r}=1-\alpha_{\rm r}(2-\alpha_{\rm r}\lambda_{\max}^{\rm r})\sigma_{\min}^2(\mbf A).$$
\end{theorem}
\begin{proof} It follows from $\mbf A\mbf x_\star^0=\mbf b$ and
	\beq\label{rec}\mbf x^k-\mbf x_\star^0
= \mbf x^{k-1}-\mbf x_\star^0-\alpha_{\rm r}\mbf A^\top\mbf S\mbf S^\top(\mbf A\mbf x^{k-1}-\mbf b)\eeq that
\begin{align*}
\| \mbf x^k-\mbf x_\star^0\|^2
&=\| \mbf x^{k-1}-\mbf x_\star^0\|^2-2\alpha_{\rm r}(\mbf x^{k-1}-\mbf x_\star^0)^\top\mbf A^\top \mbf S\mbf S^\top \mbf A(\mbf x^{k-1}-\mbf x_\star^0)\\
&\ \quad + \alpha_{\rm r}^2\|\mbf A^\top\mbf S\mbf S^\top\mbf A(\mbf x^{k-1}-\mbf x_\star^0)\|^2.
\end{align*}
Note that for any $\mbf H\succeq\mbf 0$, it holds $\lambda_{\max}(\mbf H)\mbf H\succeq\mbf H^2$. Then we have
\begin{align}
\|\mbf A^\top\mbf S\mbf S^\top\mbf A(\mbf x^{k-1}-\mbf x_\star^0)\|^2&=(\mbf x^{k-1}-\mbf x_\star^0)^\top (\mbf A^\top\mbf S\mbf S^\top\mbf A)^2 (\mbf x^{k-1}-\mbf x_\star^0)\nn\\
&\leq \lambda_{\max}(\mbf A^\top\mbf S\mbf S^\top\mbf A) (\mbf x^{k-1}-\mbf x_\star^0)^\top \mbf A^\top\mbf S\mbf S^\top\mbf A (\mbf x^{k-1}-\mbf x_\star^0)\nn\\ 
&\leq\lambda_{\max}^{\rm r}(\mbf x^{k-1}-\mbf x_\star^0)^\top \mbf A^\top\mbf S\mbf S^\top\mbf A (\mbf x^{k-1}-\mbf x_\star^0).\label{no2}
\end{align}
By $\mbf x^0-\mbf x_\star^0=\mbf A^\dag(\mbf A\mbf x^0-\mbf b)\in\ran(\mbf A^\top)$, $\mbf A^\top\mbf S\mbf S^\top\mbf A(\mbf x^{k-1}-\mbf x_\star^0)\in\ran(\mbf A^\top)$, and (\ref{rec}), we can prove that $\mbf x^k-\mbf x_\star^0\in\ran(\mbf A^\top)$ by induction. Therefore, 
\begin{align*}
\mbbe_{k-1}\bem\| \mbf x^k-\mbf x_\star^0\|^2\eem
&\leq\| \mbf x^{k-1}-\mbf x_\star^0\|^2-2\alpha_{\rm r}(\mbf x^{k-1}-\mbf x_\star^0)^\top\mbf A^\top\mbf A (\mbf x^{k-1}-\mbf x_\star^0)\\
&\ \quad + \alpha_{\rm r}^2\lambda_{\max}^{\rm r}(\mbf x^{k-1}-\mbf x_\star^0)^\top \mbf A^\top\mbf A(\mbf x^{k-1}-\mbf x_\star^0)\\
&\leq(1-\alpha_{\rm r}(2-\alpha_{\rm r}\lambda_{\max}^{\rm r})\sigma_{\min}^2(\mbf A))\| \mbf x^{k-1}-\mbf x_\star^0\|^2.
\end{align*} In the last inequality, we use the facts that $-\alpha_{\rm r}(2-\alpha_{\rm r}\lambda_{\max}^{\rm r})<0$, and for all $\mbf u\in\ran(\mbf A^\top)$, it holds $\mbf u^\top\mbf A^\top\mbf A\mbf u\geq\sigma_{\min}^2(\mbf A)\|\mbf u\|^2$.
Next, by the law of total expectation, we have $$\mbbe\bem\| \mbf x^k-\mbf x_\star^0\|^2\eem\leq\eta_{\rm r}\mbbe\bem\| \mbf x^{k-1}-\mbf x_\star^0\|^2\eem.$$
Unrolling the recurrence yields the result.
\end{proof}

According to Theorem \ref{rslt}, the best convergence rate ($\eta_{\rm r}=1-\sigma_{\min}^2(\mbf A)/\lambda_{\max}^{\rm r}$) of the BRSI algorithm is achieved when $\alpha_{\rm r}=1/\lambda_{\max}^{\rm r}$. However, the proof of Theorem \ref{rslt} uses the worst-case estimates, so the convergence bound may be pessimistic, and in practical applications it may not precisely measure the actual convergence rate of the BRSI algorithm. 

\begin{theorem}\label{rslti} Assume that $\ve>0$ and $0<\alpha_{\rm r}<\dsp\frac{2}{(1+\ve)\lambda_{\max}^{\rm r}}$. If $\bf Ax=b$ is inconsistent, then for arbitrary $\mbf x^0\in\mbbr^n$, the $k$th iterate $\mbf x^k$ of the {\rm BRSI} algorithm satisfies
$$\mbbe\bem \|\mbf x^k-\mbf x_\star^0\|^2\eem\leq \eta_\ve^k\|\mbf x^0-\mbf x_\star^0\|^2+\frac{\alpha_{\rm r}(1+1/\ve)\gamma(1-\eta_\ve^k)}{(2-\alpha_{\rm r}(1+\ve)\lambda_{\max}^{\rm r})\sigma_{\min}^2(\mbf A)},$$
where  $$\eta_\ve=1-\alpha_{\rm r}(2-\alpha_{\rm r}(1+\ve)\lambda_{\max}^{\rm r})\sigma_{\min}^2(\mbf A)$$ and $$\gamma=\mbbe\bem\|\mbf A^\top\mbf S\mbf S^\top(\mbf A\mbf A^\dag\mbf b-\mbf b)\|^2\eem.$$
\end{theorem}

\begin{proof} It follows from
	$$\mbf x^k-\mbf x_\star^0=\mbf x^{k-1}-\mbf x_\star^0-\alpha_{\rm r} \mbf A^\top\mbf S\mbf S^\top(\mbf A\mbf x^{k-1}-\mbf b)$$ that
\begin{align}
\|\mbf x^k-\mbf x_\star^0\|^2 
& =\|\mbf x^{k-1}-\mbf x_\star^0\|^2-2\alpha_{\rm r}(\mbf x^{k-1}-\mbf x_\star^0)^\top \mbf A^\top\mbf S\mbf S^\top(\mbf A\mbf x^{k-1}-\mbf b)\nn \\ 
&\quad +\alpha_{\rm r}^2\| \mbf A^\top\mbf S\mbf S^\top(\mbf A\mbf x^{k-1}-\mbf b)\|^2 \label{nsl2}\end{align}
By $\mbf A\mbf x_\star^0=\mbf A\mbf A^\dag\mbf b$, triangle inequality, and Young's inequality, we have
\begin{align} &\ \quad \| \mbf A^\top\mbf S\mbf S^\top(\mbf A\mbf x^{k-1}-\mbf b)\|^2= \| \mbf A^\top\mbf S\mbf S^\top(\mbf A\mbf x^{k-1}-\mbf A\mbf x_\star^0+\mbf A\mbf A^\dag\mbf b-\mbf b)\|^2\nn \\
&\leq (\| \mbf A^\top\mbf S\mbf S^\top\mbf A(\mbf x^{k-1}-\mbf x_\star^0)\|+\| \mbf A^\top\mbf S\mbf S^\top(\mbf A\mbf A^\dag\mbf b-\mbf b)\|)^2\nn \\
&\leq (1+\ve)\| \mbf A^\top\mbf S\mbf S^\top\mbf A(\mbf x^{k-1}-\mbf x_\star^0)\|^2+(1+1/\ve)\| \mbf A^\top\mbf S\mbf S^\top(\mbf A\mbf A^\dag\mbf b-\mbf b)\|^2.\label{last}
\end{align}
By (\ref{no2}), (\ref{nsl2}), (\ref{last}), and $\mbf A^\top\mbf A\mbf A^\dag\mbf b=\mbf A^\top\mbf b$, we have
\begin{align*}
&\ \quad \mbbe_{k-1}\bem\|\mbf x^k-\mbf x_\star^0\|^2\eem \\
& =\|\mbf x^{k-1}-\mbf x_\star^0\|^2-2\alpha_{\rm r}(\mbf x^{k-1}-\mbf x_\star^0)^\top	 \mbf A^\top \mbf A(\mbf x^{k-1}-\mbf x_\star^0) +\alpha_{\rm r}^2\mbbe_{k-1}\bem\| \mbf A^\top\mbf S\mbf S^\top(\mbf A\mbf x^{k-1}-\mbf b)\|^2\eem\\
& \leq\|\mbf x^{k-1}-\mbf x_\star^0\|^2-(2\alpha_{\rm r}-\alpha_{\rm r}^2(1+\ve)\lambda_{\max}^{\rm r})(\mbf x^{k-1}-\mbf x_\star^0)^\top	 \mbf A^\top \mbf A(\mbf x^{k-1}-\mbf x_\star^0) +\alpha_{\rm r}^2(1+1/\ve)\gamma\\
&\leq (1-\alpha_{\rm r}(2-\alpha_{\rm r}(1+\ve)\lambda_{\max}^{\rm r})\sigma_{\min}^2(\mbf A))\|\mbf x^{k-1}-\mbf x_\star^0\|^2  +\alpha_{\rm r}^2(1+1/\ve)\gamma\\
&= \eta_\ve\|\mbf x^{k-1}-\mbf x_\star^0\|^2 +\alpha_{\rm r}^2(1+1/\ve)\gamma. 
\end{align*}
Then the expected squared norm of the error can be bounded by 
\begin{align*}
\mbbe\bem\|\mbf x^k-\mbf x^0_\star\|^2\eem &\leq\eta_\ve\mbbe\bem\|\mbf x^{k-1}-\mbf x_\star^0\|^2\eem +\alpha_{\rm r}^2(1+1/\ve)\gamma \\
&\leq \eta_\ve^k\|\mbf x^0-\mbf x^0_\star\|^2+\alpha_{\rm r}^2(1+1/\ve)\gamma\sum_{i=0}^{k-1}\eta_\ve^i\\ 
&= \eta_\ve^k\|\mbf x^0-\mbf x^0_\star\|^2+\frac{\alpha_{\rm r}^2(1+1/\ve)\gamma(1-\eta_\ve^k)}{1-\eta_\ve}\\ 
&=\eta_\ve^k\|\mbf x^0-\mbf x^0_\star\|^2+\frac{\alpha_{\rm r}(1+1/\ve)\gamma(1-\eta_\ve^k)}{(2-\alpha_{\rm r}(1+\ve)\lambda_{\max}^{\rm r})\sigma_{\min}^2(\mbf A)}.
\end{align*} This completes the proof.
\end{proof} 

\subsubsection{The randomized Kaczmarz algorithm} The RK algorithm \cite{strohmer2009rando} is one special case of the BRSI algorithm. Choosing $\dsp\mbf S=\frac{\|\mbf A\|_\rmf}{\|\mbf A_{i,:}\|}\mbf I_{:,i}$ with probability $\dsp\frac{\|\mbf A_{i,:}\|^2}{\|\mbf A\|_\rmf^2}$ in (\ref{brs}), we have $$\mbbe\bem\mbf S\mbf S^\top\eem=\|\mbf A\|_\rmf^2\mbbe\bem\dsp\frac{\mbf I_{:,i}(\mbf I_{:,i})^\top}{\|\mbf A_{i,:}\|^2}\eem= \|\mbf A\|_\rmf^2\sum_{i=1}^m \frac{ \mbf I_{:,i}(\mbf I_{:,i})^\top}{\|\mbf A_{i,:}\|^2}\frac{\|\mbf A_{i,:}\|^2}{\|\mbf A\|_\rmf^2}=\sum_{i=1}^m \mbf I_{:,i}(\mbf I_{:,i})^\top=\mbf I,$$ and recover the RK iteration \beq\label{rk}\mbf x^k=\mbf x^{k-1}-\alpha_{\rm r}\|\mbf A\|_\rmf^2 \frac{\mbf A_{i,:}\mbf x^{k-1}-\mbf b_i}{\|\mbf A_{i,:}\|^2}(\mbf A_{i,:})^\top.\eeq For this case, we have $\lambda_{\max}^{\rm r}=\|\mbf A\|_\rmf^2$. Choosing $\alpha_{\rm r}={1}/{\|\mbf A\|_\rmf^2}$ in Theorem \ref{rslt} yields the convergence estimate of \cite{strohmer2009rando}: $$\mbbe\bem \|\mbf x^k-\mbf x_\star^0\|^2\eem\leq \l(1-\frac{\sigma_{\min}^2(\mbf A)}{\|\mbf A\|_\rmf^2}\r)^k\|\mbf x^0-\mbf x_\star^0\|^2.$$

\subsubsection{The block row uniform sampling algorithm} We propose one special case of the BRSI algorithm by using uniform sampling and refer to it as the block row uniform sampling (BRUS) algorithm. We note that the BRUS algorithm is also one special case of the randomized average block Kaczmarz algorithm \cite{necoara2019faste}. Assume $1\leq\ell\leq m$. Let $\mcali$ denote the set consisting of the uniform sampling of $\ell$ different numbers of $[m]$. Setting $\mbf S=\sqrt{m/\ell}\mbf I_{:,\mcali}$, we have $$\mbbe\bem\mbf S\mbf S^\top\eem=\frac{\dsp\frac{m}{\ell}}{\begin{pmatrix}m\\ \ell\end{pmatrix}}\sum_{\mcali\subseteq[m],\ |\mcali|=\ell}\mbf I_{:,\mcali}\mbf I_{:,\mcali}^\top=\frac{\dsp\frac{m}{\ell}}{\begin{pmatrix}m\\ \ell\end{pmatrix}}\begin{pmatrix}m-1\\ \ell-1\end{pmatrix}\mbf I=\mbf I,$$ and obtain the iteration $$\mbf x^k=\mbf x^{k-1}-\alpha_{\rm r}\frac{m}{\ell}(\mbf A_{\mcali,:})^\top(\mbf A_{\mcali,:}\mbf x^{k-1}-\mbf b_\mcali).$$ For this case, we have $$\lambda_{\max}^{\rm r}=\frac{m}{\ell}\max_{\mcali\subseteq[m],|\mcali|=\ell}\|\mbf A_{\mcali,:}\|^2.$$ By Theorem \ref{rslt}, the BRUS algorithm can have a faster convergence rate than that of the RK algorithm if there exists $\ell\in[m]$ satisfying $$\frac{m}{\ell}\max_{\mcali\subseteq[m],|\mcali|=\ell}\|\mbf A_{\mcali,:}\|^2\leq\|\mbf A\|_\rmf^2.$$

We present the details of the BRUS algorithm with block size $\ell$ and initial guess $\mbf x^0=\mbf 0$ in Algorithm 1. We also note that the constant $m/\ell$ is incorporated in the stepsize parameter $\alpha_{\rm r}$.

\begin{center}
\begin{tabular*}{160mm}{l}
\toprule {\bf Algorithm 1:} BRUS($\ell$)\\ 
\hline \noalign{\smallskip}
\qquad Initialize $\mbf x^0=\mbf 0$ and a fixed $1\leq\ell\leq m$\\
\qquad {\bf for} $k=1,2,\ldots $ {\bf do}\\
\qquad \qquad  Select randomly a set $\mcali$ consisting of the uniform sampling of $\ell$ numbers of $[m]$\\
\qquad \qquad  Update $\mbf x^k=\mbf x^{k-1}-\alpha_{\rm r}(\mbf A_{\mcali,:})^\top(\mbf A_{\mcali,:}\mbf x^{k-1}-\mbf b_\mcali)$ \\
\qquad {\bf end for} \\ \bottomrule
\end{tabular*}
\end{center}

\subsection{Block column sampling} In this subsection, we consider the case $\mbf S=\mbf I$  and refer to the resulting algorithm as the  block column sampling iterative (BCSI) algorithm. Given an arbitrary initial guess $\mbf x^0\in\mbbr^n$, the $k$th iterate of the BCSI algorithm is  \beq\label{bcsi}\mbf x^k=\mbf x^{k-1}-\alpha_{\rm c}\mbf T\mbf T^\top\mbf A^\top(\mbf A\mbf x^{k-1}-\mbf b),\eeq where the stepsize parameter $\alpha_{\rm c}>0$, and the random parameter matrix $\mbf T$ is sampled independently in each iteration from a distribution $\mcald_{\rm c}$ and satisfies $\mbbe\bem \mbf T\mbf T^\top \eem=\mbf I.$ 

In the following, we shall present the convergence of $\mbbe\bem\|\mbf A(\mbf x^k-\mbf A^\dag\mbf b)\|^2\eem$ for arbitrary linear systems. Throughout, we define $$\lambda_{\max}^{\rm c}=\max_{\mbf T\sim\mcald_{\rm c}}\lambda_{\max}(\mbf A\mbf T\mbf T^\top\mbf A^\top).$$ 
\begin{theorem}\label{cslt} Assume that $0<\alpha_{\rm c}<2/\lambda_{\max}^{\rm c}$. 
For arbitrary $\mbf x^0\in\mbbr^n$, the $k$th iterate $\mbf x^k$ of the {\rm BCSI} algorithm satisfies
$$\mbbe\bem \|\mbf A(\mbf x^k-\mbf A^\dag\mbf b)\|^2\eem\leq \eta_{\rm c}^k\|\mbf A(\mbf x^0-\mbf A^\dag\mbf b)\|^2,$$ where $$\eta_{\rm c}=1-\alpha_{\rm c}(2-\alpha_{\rm c}\lambda_{\max}^{\rm c})\sigma_{\min}^2(\mbf A).$$
\end{theorem}

\begin{proof} It follows from (\ref{bcsi}) and $\mbf A^\top\mbf A\mbf A^\dag\mbf b=\mbf A^\top\mbf b$ that
	\begin{align*}\mbf A(\mbf x^k-\mbf A^\dag\mbf b)
&=\mbf A(\mbf x^{k-1}-\mbf A^\dag\mbf b-\alpha_{\rm c}\mbf T\mbf T^\top\mbf A^\top(\mbf A\mbf x^{k-1}-\mbf b))\\
&=\mbf A(\mbf x^{k-1}-\mbf A^\dag\mbf b)-\alpha_{\rm c}\mbf A\mbf T\mbf T^\top\mbf A^\top\mbf A(\mbf x^{k-1}-\mbf A^\dag\mbf b).
\end{align*} Then we have
\begin{align}
\|\mbf A(\mbf x^k-\mbf A^\dag\mbf b)\|^2
&=\|\mbf A(\mbf x^{k-1}-\mbf A^\dag\mbf b)\|^2-2\alpha_{\rm c}(\mbf x^{k-1}-\mbf A^\dag\mbf b)^\top\mbf A^\top\mbf A\mbf T\mbf T^\top\mbf A^\top\mbf A(\mbf x^{k-1}-\mbf A^\dag\mbf b)\nn \\
&\ \quad + \alpha_{\rm c}^2(\mbf x^{k-1}-\mbf A^\dag\mbf b)^\top\mbf A^\top(\mbf A\mbf T\mbf T^\top\mbf A^\top)^2\mbf A(\mbf x^{k-1}-\mbf A^\dag\mbf b).\label{add1}
\end{align}
Note that
\begin{align}
&\ \quad (\mbf x^{k-1}-\mbf A^\dag\mbf b)^\top\mbf A^\top(\mbf A\mbf T\mbf T^\top\mbf A^\top)^2\mbf A(\mbf x^{k-1}-\mbf A^\dag\mbf b)\nn \\
&\leq \lambda_{\max}(\mbf A\mbf T\mbf T^\top\mbf A^\top) (\mbf x^{k-1}-\mbf A^\dag\mbf b)^\top\mbf A^\top\mbf A\mbf T\mbf T^\top\mbf A^\top\mbf A(\mbf x^{k-1}-\mbf A^\dag\mbf b)\nn\\
&\leq \lambda_{\max}^{\rm c} (\mbf x^{k-1}-\mbf A^\dag\mbf b)^\top\mbf A^\top\mbf A\mbf T\mbf T^\top\mbf A^\top\mbf A(\mbf x^{k-1}-\mbf A^\dag\mbf b).\label{add2}
\end{align}
By (\ref{add1}), (\ref{add2}), and $\mbf A(\mbf x^{k-1}-\mbf A^\dag\mbf b)\in\ran(\mbf A)$, we have
\begin{align*}
\mbbe_{k-1}\bem\|\mbf A(\mbf x^k-\mbf A^\dag\mbf b)\|^2\eem
&\leq\|\mbf A(\mbf x^{k-1}-\mbf A^\dag\mbf b)\|^2-2\alpha_{\rm c}(\mbf x^{k-1}-\mbf A^\dag\mbf b)^\top\mbf A^\top\mbf A\mbf A^\top\mbf A(\mbf x^{k-1}-\mbf A^\dag\mbf b)\\
&\ \quad + \alpha_{\rm c}^2\lambda_{\max}^{\rm c}(\mbf x^{k-1}-\mbf A^\dag\mbf b)^\top\mbf A^\top\mbf A\mbf A^\top\mbf A(\mbf x^{k-1}-\mbf A^\dag\mbf b)\\
&\leq(1-\alpha_{\rm c}(2-\alpha_{\rm c}\lambda_{\max}^{\rm c})\sigma_{\min}^2(\mbf A))\|\mbf A(\mbf x^{k-1}-\mbf A^\dag\mbf b)\|^2.
\end{align*} In the last inequality, we use the facts that $-\alpha_{\rm c}(2-\alpha_{\rm c}\lambda_{\max}^{\rm c})<0$, and for all $\mbf u\in\ran(\mbf A)$, it holds $\mbf u^\top\mbf A\mbf A^\top\mbf u\geq\sigma_{\min}^2(\mbf A)\|\mbf u\|^2$.
Next, by the law of total expectation, we have $$\mbbe\bem\|\mbf A(\mbf x^k-\mbf A^\dag\mbf b)\|^2\eem\leq\eta_{\rm c}\mbbe\bem\|\mbf A(\mbf x^{k-1}-\mbf A^\dag\mbf b)\|^2\eem$$
Unrolling the recurrence yields the result.
\end{proof}

\begin{remark} If $\mbf A$ has full column rank, then Theorem \ref{cslt} implies that $\mbf x^k$ of the {\rm BCSI} algorithm converges linearly to $\mbf A^\dag\mbf b$ in the mean square sense.
\end{remark}

\subsubsection{The randomized coordinate descent algorithm} The RCD algorithm \cite{leventhal2010rando} is one special case of the BCSI algorithm. Choosing $\dsp\mbf T=\frac{\|\mbf A\|_\rmf}{\|\mbf A_{:,j}\|}\mbf I_{:,j}$ with probability $\dsp\frac{\|\mbf A_{:,j}\|^2}{\|\mbf A\|_\rmf^2}$ in (\ref{bcsi}), we have $$\mbbe\bem\mbf T\mbf T^\top\eem=\|\mbf A\|_\rmf^2\mbbe\bem\dsp\frac{\mbf I_{:,j}(\mbf I_{:,j})^\top}{\|\mbf A_{:,j}\|^2}\eem= \|\mbf A\|_\rmf^2\sum_{j=1}^n \frac{ \mbf I_{:,j}(\mbf I_{:,j})^\top}{\|\mbf A_{:,j}\|^2}\frac{\|\mbf A_{:,j}\|^2}{\|\mbf A\|_\rmf^2}=\sum_{j=1}^n\mbf I_{:,j}(\mbf I_{:,j})^\top=\mbf I,$$ and recover the RCD iteration $$\mbf x^k=\mbf x^{k-1}-\alpha_{\rm c}\|\mbf A\|_\rmf^2 \frac{(\mbf A_{:,j})^\top(\mbf A\mbf x^{k-1}-\mbf b)}{\|\mbf A_{:,j}\|^2}\mbf I_{:,j}.$$ For this case, we have $\lambda_{\max}^{\rm c}=\|\mbf A\|_\rmf^2$. Choosing $\alpha_{\rm c}=1/\|\mbf A\|_\rmf^2$ in Theorem \ref{cslt} yields the convergence estimate of \cite{leventhal2010rando,ma2015conve}: $$\mbbe\bem \|\mbf A(\mbf x^k-\mbf A^\dag\mbf b)\|^2\eem\leq \l(1-\frac{\sigma_{\min}^2(\mbf A)}{\|\mbf A\|_\rmf^2}\r)^k\|\mbf A(\mbf x^0-\mbf A^\dag\mbf b)\|^2.$$

\subsubsection{The block column uniform sampling algorithm} We propose one new special case of the BCSI algorithm by using uniform sampling and refer to it as the block column uniform sampling (BCUS) algorithm. Assume $1\leq\ell\leq n$. Let $\mcalj$ denote the set consisting of the uniform sampling of $\ell$ different numbers of $[n]$. Setting $\mbf T=\sqrt{n/\ell}\mbf I_{:,\mcalj}$, we have $$\mbbe\bem\mbf T\mbf T^\top\eem=\frac{\dsp\frac{n}{\ell}}{\begin{pmatrix}n\\ \ell\end{pmatrix}}\sum_{\mcalj\subseteq[n],\ |\mcalj|=\ell}\mbf I_{:,\mcalj}\mbf I_{:,\mcalj}^\top=\frac{\dsp\frac{n}{\ell}}{\begin{pmatrix}n\\ \ell\end{pmatrix}}\begin{pmatrix}n-1\\ \ell-1\end{pmatrix}\mbf I=\mbf I,$$ and obtain the iteration $$\mbf x^k=\mbf x^{k-1}-\alpha_{\rm c}\frac{n}{\ell}\mbf I_{:,\mcalj}(\mbf A_{:,\mcalj})^\top(\mbf A\mbf x^{k-1}-\mbf b).$$ For this case, we have $$\lambda_{\max}^{\rm c}=\frac{n}{\ell}\max_{\mcalj\subseteq[n], |\mcalj|=\ell}\|\mbf A_{:,\mcalj}\|^2.$$ By Theorem \ref{cslt}, the BCUS algorithm can have a faster convergence rate than that of the RCD algorithm if there exists $\ell\in[n]$ satisfying $$\frac{n}{\ell}\max_{\mcalj\subseteq[n], |\mcalj|=\ell}\|\mbf A_{:,\mcalj}\|^2\leq\|\mbf A\|_\rmf^2.$$

To avoid entire matrix-vector multiplications, we introduce an auxiliary vector $\mbf r^k=\mbf b-\mbf A\mbf x^k$ in each iteration of the BCUS algorithm. We present the details of the BCUS algorithm with block size $\ell$ and initial guess $\mbf x^0=\mbf 0$ in Algorithm 2. We also note that the constant $n/\ell$ is incorporated in the stepsize parameter $\alpha_{\rm c}$. 
\begin{center}
\begin{tabular*}{160mm}{l}
\toprule {\bf Algorithm 2:} BCUS($\ell$)\\ 
\hline \noalign{\smallskip}
\qquad Initialize $\mbf x^0=\mbf 0$, $\mbf r^0=\mbf b$, and a fixed $1\leq\ell\leq n$\\
\qquad {\bf for} $k=1,2,\ldots $ {\bf do}\\
\qquad \qquad  Select randomly a set $\mcalj$ consisting of the uniform sampling of $\ell$ numbers of $[n]$\\
\qquad \qquad  Compute $\mbf w^k=\alpha_{\rm c}(\mbf A_{:,\mcalj})^\top\mbf r^{k-1}$\\
\qquad \qquad  Update $\mbf x_\mcalj^k=\mbf x_\mcalj^{k-1}+\mbf w^k$ and $\mbf r^k=\mbf r^{k-1}-\mbf A_{:,\mcalj}\mbf w^k.$\\
\qquad {\bf end for} \\ \bottomrule
\end{tabular*}
\end{center}

\section{The extended block row sampling iterative algorithm}\label{eb}
 
Solving $\bf A^\top z=0$ by the RK algorithm with initial guess $\mbf z^0\in\mbf b+\ran(\mbf A)$ produces a sequence $\{\mbf z^k\}$, which converges to $\mbf b-\mbf A\mbf A^\dag\mbf b$ (see, e.g., \cite{du2019tight}). Zouzias and Freris \cite{zouzias2013rando} proved that the $k$th iterate $\mbf x^k$ (which is produced by one RK update for ${\bf Ax=b-z}^k$ from $\mbf x^{k-1}$) of the REK algorithm converges to a solution of $\bf Ax=AA^\dag b$. We note that any solution of $\bf Ax=AA^\dag b$ is a solution of $\bf Ax=b$ if it is consistent or a least squares solution of $\bf Ax=b$ if it is inconsistent. In this section, based on the idea of the REK algorithm, we propose an extended block row sampling iterative (EBRSI) algorithm. In Appendix \ref{appendix}, we also propose an extended block column and row sampling iterative algorithm based on the idea of the randomized extended Gauss--Seidel algorithm \cite{ma2015conve}. 

Given $\mbf z^0\in\mbf b+\ran(\mbf A)$ and an arbitrary initial guess $\mbf x^0\in\mbbr^n$, the iterates of the EBRSI algorithm at step $k$ are defined as  
\begin{align}
\mbf z^k & =\mbf z^{k-1}-\alpha_{\rm c}\mbf A\mbf T\mbf T^\top\mbf A^\top\mbf z^{k-1},\label{eslz}\\
\mbf x^k & =\mbf x^{k-1}-\alpha_{\rm r}\mbf A^\top\mbf S\mbf S^\top(\mbf A\mbf x^{k-1}-\mbf b+\mbf z^k),\label{eslx}
	\end{align}  where the random parameter matrices $\mbf S$ and $\mbf T$ are independent, and satisfy $$\mbbe\bem \mbf S\mbf S^\top \eem=\mbf I,\qquad \mbbe\bem \mbf T\mbf T^\top \eem=\mbf I.$$ We note that the iteration (\ref{eslz}) is the BRSI algorithm for $\bf A^\top z=0$ with initial guess $\mbf z^0\in\mbf b+\ran(\mbf A)$, and the iterate $\mbf x^k$ in (\ref{eslx}) is one BRSI update for ${\bf Ax=b-z}^k$ from $\mbf x^{k-1}$. By Theorem \ref{rslt}, we have \beq\label{sun1}\mbbe\bem \|\mbf z^k-(\mbf I-\mbf A\mbf A^\dag)\mbf b\|^2\eem\leq \eta_{\rm c}^k\|\mbf z^0-(\mbf I-\mbf A\mbf A^\dag)\mbf b\|^2.\eeq

In the following, we shall present two convergence results of the EBRSI algorithm: Theorem \ref{seli} is on the convergence of $\|\mbbe\bem \mbf x^k\eem-\mbf x_\star^0\|$, and Theorem \ref{sel} is on the convergence of $\mbbe\bem\|\mbf x^k-\mbf x_\star^0\|^2\eem$. We emphasize that both the convergence results hold for arbitrary linear systems. Let $\mbbe_{k-1}\bem\cdot\eem$ denote the conditional expectation conditioned on $\mbf z^{k-1}$ and $\mbf x^{k-1}$. Let $\mbbe_{k-1}^{\rm r}\bem\cdot\eem$ denote the conditional expectation conditioned on $\mbf z^k$ and $\mbf x^{k-1}$. Then, by the law of total expectation, we have $$\mbbe_{k-1}\bem\cdot\eem=\mbbe_{k-1}\bem\mbbe_{k-1}^{\rm r}\bem\cdot\eem\eem.$$

\begin{theorem}\label{seli}
 For arbitrary $\mbf z^0\in\mbbr^m$ and $\mbf x^0\in\mbbr^n$, the $k$th iterate $\mbf x^k$ of the {\rm EBRSI} algorithm satisfies $$\mbbe\bem \mbf x^k-\mbf x_\star^0\eem=(\mbf I-\alpha_{\rm r}\mbf A^\top\mbf A)^k(\mbf x^0-\mbf x_\star^0)-\alpha_{\rm r}\sum_{i=0}^{k-1}(\mbf I-\alpha_{\rm r}\mbf A^\top\mbf A)^i(\mbf I-\alpha_{\rm c}\mbf A^\top\mbf A)^{k-i}\mbf A^\top\mbf z^0.$$ Moreover, \beq\label{reke}
 	\|\mbbe\bem\mbf x^k\eem-\mbf x_\star^0\|\leq\delta^k(\|\mbf x^0-\mbf x_\star^0\|+k\alpha_{\rm r}\|\mbf A^\top\mbf z^0\|),  
 \eeq
 where $$\delta=\max_{1\leq i\leq r}\{|1-\alpha_{\rm r}\sigma_i^2(\mbf A)|,|1-\alpha_{\rm c}\sigma_i^2(\mbf A)|\}.$$
\end{theorem}

\begin{proof}  By (\ref{eslz}), we have
	$$\mbbe_{k-1}\bem\mbf z^k\eem=\mbf z^{k-1}-\alpha_{\rm c}\mbf A\mbf A^\top\mbf z^{k-1}=(\mbf I-\alpha_{\rm c}\mbf A\mbf A^\top)\mbf z^{k-1},$$ which, by the law of total expectation, yields $$\mbbe\bem\mbf z^k\eem=(\mbf I-\alpha_{\rm c}\mbf A\mbf A^\top)\mbbe\bem\mbf z^{k-1}\eem=\cdots=(\mbf I-\alpha_{\rm c}\mbf A\mbf A^\top)^k \mbf z^0.$$ Taking expectation conditioned on $\mbf z^{k-1}$ and $\mbf x^{k-1}$ for $$\mbf x^k-\mbf x_\star^0 =\mbf x^{k-1}-\mbf x_\star^0-\alpha_{\rm r}\mbf A^\top\mbf S\mbf S^\top(\mbf A\mbf x^{k-1}-\mbf b+\mbf z^k),$$ we obtain
\begin{align*}\mbbe_{k-1}\bem\mbf x^k-\mbf x_\star^0\eem  
& =\mbbe_{k-1}\bem\mbbe_{k-1}^{\rm r}\bem \mbf x^{k-1}-\mbf x_\star^0 -\alpha_{\rm r}\mbf A^\top\mbf S\mbf S^\top(\mbf A\mbf x^{k-1}-\mbf b+\mbf z^k)\eem\eem\\
& =\mbbe_{k-1}\bem\mbf x^{k-1}-\mbf x_\star^0 -\alpha_{\rm r}\mbf A^\top(\mbf A\mbf x^{k-1}-\mbf b+\mbf z^k)\eem\\
& =\mbf x^{k-1}-\mbf x_\star^0 -\alpha_{\rm r}(\mbf A^\top\mbf A\mbf x^{k-1}-\mbf A^\top\mbf b) -\alpha_{\rm r}\mbf A^\top\mbbe_{k-1}\bem\mbf z^k\eem\\
& =\mbf x^{k-1}-\mbf x_\star^0 -\alpha_{\rm r}(\mbf A^\top\mbf A\mbf x^{k-1}-\mbf A^\top\mbf A\mbf x_\star^0) -\alpha_{\rm r}\mbf A^\top\mbbe_{k-1}\bem\mbf z^k\eem\\
& =(\mbf I-\alpha_{\rm r}\mbf A^\top\mbf A)(\mbf x^{k-1}-\mbf x_\star^0)  -\alpha_{\rm r}\mbf A^\top\mbbe_{k-1}\bem\mbf z^k\eem,
\end{align*}
which, by the law of total expectation, yields 
\begin{align*}
	\mbbe\bem\mbf x^k-\mbf x_\star^0\eem
	&=(\mbf I-\alpha_{\rm r}\mbf A^\top\mbf A)\mbbe\bem\mbf x^{k-1}-\mbf x_\star^0\eem -\alpha_{\rm r}\mbf A^\top\mbbe\bem\mbf z^k\eem\\
	&=(\mbf I-\alpha_{\rm r}\mbf A^\top\mbf A)\mbbe\bem\mbf x^{k-1}-\mbf x_\star^0\eem -\alpha_{\rm r}\mbf A^\top(\mbf I-\alpha_{\rm c}\mbf A\mbf A^\top)^k\mbf z^0\\
	&=(\mbf I-\alpha_{\rm r}\mbf A^\top\mbf A)\mbbe\bem\mbf x^{k-1}-\mbf x_\star^0\eem -\alpha_{\rm r}(\mbf I-\alpha_{\rm c}\mbf A^\top\mbf A)^k\mbf A^\top\mbf z^0\\
	&=(\mbf I-\alpha_{\rm r}\mbf A^\top\mbf A)^2\mbbe\bem\mbf x^{k-2}-\mbf x_\star^0\eem -\alpha_{\rm r}(\mbf I-\alpha_{\rm r}\mbf A^\top\mbf A)(\mbf I-\alpha_{\rm c}\mbf A^\top\mbf A)^{k-1}\mbf A^\top\mbf z^0\\
	&\quad \ -\alpha_{\rm r}(\mbf I-\alpha_{\rm c}\mbf A^\top\mbf A)^k\mbf A^\top\mbf z^0\\
	&=\cdots\\
	&=(\mbf I-\alpha_{\rm r}\mbf A^\top\mbf A)^k(\mbf x^0-\mbf x_\star^0)-\alpha_{\rm r}\sum_{i=0}^{k-1}(\mbf I-\alpha_{\rm r}\mbf A^\top\mbf A)^i(\mbf I-\alpha_{\rm c}\mbf A^\top\mbf A)^{k-i}\mbf A^\top\mbf z^0.
\end{align*} Taking 2-norm, by triangle inequality, $\mbf x^0-\mbf x_\star^0\in\ran(\mbf A^\top)$, $\mbf A^\top\mbf z^0\in\ran(\mbf A^\top)$, and Lemma \ref{leqd}, we obtain the estimate (\ref{reke}).
\end{proof}

\begin{remark}
	In Theorem \ref{seli}, no assumptions about the dimensions or rank of $\mbf A$ are assumed, and the system $\bf Ax=b$ can be consistent or inconsistent. If $0<\alpha_{\rm r}<2/\sigma_{\max}^2(\mbf A)$ and $0<\alpha_{\rm c}<2/\sigma_{\max}^2(\mbf A)$, then $0<\delta<1$. This means $\mbf x^k$ is an asymptotically unbiased estimator for $\mbf x_\star^0$.
\end{remark}	

\begin{theorem}\label{sel}
Assume that $0<\alpha_{\rm c}<2/\lambda_{\max}^{\rm c}$ and  $0<\alpha_{\rm r}<2/\lambda_{\max}^{\rm r}$. For arbitrary $\mbf x^0\in\mbbr^n$, $\mbf z^0\in\mbf b+\ran(\mbf A)$, and $\ve>0$, the $k$th iterate $\mbf x^k$ of the {\rm EBRSI} algorithm satisfies 
\begin{align*}
\mbbe\bem\|\mbf x^k-\mbf x_\star^0\|^2\eem &\leq(1+\ve)^k\eta_{\rm r}^k\|\mbf x^0-\mbf x_\star^0\|^2\\ 
&\quad +(1+1/\ve)\alpha_{\rm r}^2\lambda_{\max}^{\rm r}\|\mbf z^0-(\mbf I-\mbf A\mbf A^\dag)\mbf b\|^2\sum_{i=0}^{k-1}\eta_{\rm c}^{k-i}(1+\ve)^i\eta_{\rm r}^i,
\end{align*} where $$\eta_{\rm r}=1-\alpha_{\rm r}(2-\alpha_{\rm r}\lambda_{\max}^{\rm r})\sigma_{\min}^2(\mbf A),\quad \eta_{\rm c}=1-\alpha_{\rm c}(2-\alpha_{\rm c}\lambda_{\max}^{\rm c})\sigma_{\min}^2(\mbf A).$$
\end{theorem}

\begin{proof}
	We define $$\wh{\mbf x}^k=\mbf x^{k-1}-\alpha_{\rm r}\mbf A^\top\mbf S\mbf S^\top(\mbf A\mbf x^{k-1}-\mbf A\mbf x_\star^0),$$ which is one BRSI update for $\mbf A\mbf x=\mbf A\mbf x_\star^0$ from $\mbf x^{k-1}$.  We have $$\mbf x^k-\wh{\mbf x}^k=\alpha_{\rm r}\mbf A^\top\mbf S\mbf S^\top(\mbf b-\mbf A\mbf x_\star^0-\mbf z^k).$$ It follows from $\lambda_{\max}(\mbf S^\top\mbf A\mbf A^\top\mbf S)=\lambda_{\max}(\mbf A^\top\mbf S\mbf S^\top\mbf A)$ that
\begin{align*}
\|\mbf x^k-\wh{\mbf x}^k\|^2&=\alpha_{\rm r}^2(\mbf b-\mbf A\mbf x_\star^0-\mbf z^k)^\top\mbf S\mbf S^\top\mbf A\mbf A^\top\mbf S\mbf S^\top(\mbf b-\mbf A\mbf x_\star^0-\mbf z^k)\\ 
&\leq\alpha_{\rm r}^2\lambda_{\max}^{\rm r}(\mbf b-\mbf A\mbf x_\star^0-\mbf z^k)^\top\mbf S\mbf S^\top(\mbf b-\mbf A\mbf x_\star^0-\mbf z^k).
\end{align*}
Taking conditional expectation conditioned on $\mbf z^{k-1}$ and $\mbf x^{k-1}$, by $\mbf b-\mbf A\mbf x_\star^0=(\mbf I-\mbf A\mbf A^\dag)\mbf b$, we have
\begin{align*}
	\mbbe_{k-1}\bem\|\mbf x^k-\wh{\mbf x}^k\|^2\eem&=\mbbe_{k-1}\bem\mbbe_{k-1}^{\rm r}\bem\|\mbf x^k-\wh{\mbf x}^k\|^2\eem\eem\\
	&\leq \alpha_{\rm r}^2\lambda_{\max}^{\rm r}\mbbe_{k-1}\bem \|\mbf z^k-(\mbf I-\mbf A\mbf A^\dag)\mbf b\|^2\eem
\end{align*}
Then, by the law of total expectation and the estimate (\ref{sun1}), we have
\begin{align*}
\mbbe\bem\|\mbf x^k-\wh{\mbf x}^k\|^2\eem
\leq \alpha_{\rm r}^2\lambda_{\max}^{\rm r}\mbbe\bem \|\mbf z^k-(\mbf I-\mbf A\mbf A^\dag)\mbf b\|^2\eem\leq \alpha_{\rm r}^2\lambda_{\max}^{\rm r}\eta_{\rm c}^k\|\mbf z^0-(\mbf I-\mbf A\mbf A^\dag)\mbf b\|^2.	
\end{align*}
By $\mbf x^0-\mbf x_\star^0\in\ran(\mbf A^\top)$ and $\mbf A^\top\mbf S\mbf S^\top(\mbf A\mbf x^{k-1}-\mbf b+\mbf z^k)\in\ran(\mbf A^\top)$, we can show that $\mbf x^k-\mbf x_\star^0\in\ran(\mbf A^\top)$ by induction. Then,
\begin{align*}
	\|\wh{\mbf x}^k-\mbf x_\star^0\|^2 
	& = \|\mbf x^{k-1}-\mbf x_\star^0\|^2-2\alpha_{\rm r}(\mbf x^{k-1}-\mbf x_\star^0)^\top\mbf A^\top\mbf S\mbf S^\top\mbf A(\mbf x^{k-1}-\mbf x_\star^0)\\
	&\ \quad +\alpha_{\rm r}^2(\mbf x^{k-1}-\mbf x_\star^0)^\top(\mbf A^\top\mbf S\mbf S^\top\mbf A)^2(\mbf x^{k-1}-\mbf x_\star^0)\\
	&\leq \|\mbf x^{k-1}-\mbf x_\star^0\|^2-2\alpha_{\rm r}(\mbf x^{k-1}-\mbf x_\star^0)^\top\mbf A^\top\mbf S\mbf S^\top\mbf A(\mbf x^{k-1}-\mbf x_\star^0)\\
	&\ \quad +\alpha_{\rm r}^2\lambda_{\max}^{\rm r}(\mbf x^{k-1}-\mbf x_\star^0)^\top\mbf A^\top\mbf S\mbf S^\top\mbf A(\mbf x^{k-1}-\mbf x_\star^0).
\end{align*} Taking conditional expectation conditioned on $\mbf z^{k-1}$ and $\mbf x^{k-1}$, we have 
\begin{align*}\mbbe_{k-1}\bem\|\wh{\mbf x}^k-\mbf x_\star^0\|^2\eem
&\leq \|\mbf x^{k-1}-\mbf x_\star^0\|^2-2\alpha_{\rm r}(\mbf x^{k-1}-\mbf x_\star^0)^\top\mbf A^\top\mbf A(\mbf x^{k-1}-\mbf x_\star^0)\\
&\ \quad +\alpha_{\rm r}^2\lambda_{\max}^{\rm r}(\mbf x^{k-1}-\mbf x_\star^0)^\top\mbf A^\top\mbf A(\mbf x^{k-1}-\mbf x_\star^0)\\ 
&\leq \eta_{\rm r}\|\mbf x^{k-1}-\mbf x_\star^0\|^2.\end{align*} By the law of total expectation, we have $$\mbbe\bem\|\wh{\mbf x}^k-\mbf x_\star^0\|^2\eem\leq \eta_{\rm r}\mbbe\bem\|\mbf x^{k-1}-\mbf x_\star^0\|^2\eem.$$ By triangle inequality and Young's inequality, we have
$$\|\mbf x^k-\mbf x_\star^0\|^2\leq(\|\mbf x^k-\wh{\mbf x}^k\|+\|\wh{\mbf x}^k-\mbf x_\star^0\|)^2\leq(1+1/\ve)\|\mbf x^k-\wh{\mbf x}^k\|^2+(1+\ve)\|\wh{\mbf x}^k-\mbf x_\star^0\|^2.$$  Taking expectation, we have
\begin{align*}
\mbbe\bem\|\mbf x^k-\mbf x_\star^0\|^2\eem
&\leq(1+1/\ve)\mbbe\bem\|\mbf x^k-\wh{\mbf x}^k\|^2\eem+(1+\ve)\mbbe\bem\|\wh{\mbf x}^k-\mbf x_\star^0\|^2\eem\\
&\leq(1+1/\ve)\alpha_{\rm r}^2\lambda_{\max}^{\rm r}\eta_{\rm c}^k\|\mbf z^0-(\mbf I-\mbf A\mbf A^\dag)\mbf b\|^2+(1+\ve)\eta_{\rm r}\mbbe\bem\|\mbf x^{k-1}-\mbf x_\star^0\|^2\eem\\
&\leq (1+1/\ve)\alpha_{\rm r}^2\lambda_{\max}^{\rm r}\|\mbf z^0-(\mbf I-\mbf A\mbf A^\dag)\mbf b\|^2(\eta_{\rm c}^k+\eta_{\rm c}^{k-1}(1+\ve)\eta_{\rm r})\\
& \ \quad +(1+\ve)^2\eta_{\rm r}^2\mbbe\bem\|\mbf x^{k-2}-\mbf x_\star^0\|^2\eem\\
&\leq \cdots\\
&\leq (1+1/\ve)\alpha_{\rm r}^2\lambda_{\max}^{\rm r}\|\mbf z^0-(\mbf I-\mbf A\mbf A^\dag)\mbf b\|^2\sum_{i=0}^{k-1}\eta_{\rm c}^{k-i}(1+\ve)^i\eta_{\rm r}^i\\
&\ \quad +(1+\ve)^k\eta_{\rm r}^k\|\mbf x^0-\mbf x_\star^0\|^2.
\end{align*} This completes the proof.
\end{proof}

\begin{remark} 	In Theorem \ref{sel}, no assumptions about the dimensions or rank of $\mbf A$ are assumed, and the system $\bf Ax=b$ can be consistent or inconsistent.  Let $\eta=\max\{\eta_{\rm r},\eta_{\rm c}\}$. It follows from $0<\alpha_{\rm c}<2/\lambda_{\max}^{\rm c}$ and  $0<\alpha_{\rm r}<2/\lambda_{\max}^{\rm r}$ that $\eta<1$. Assume that $\ve$ satisfies $(1+\ve)\eta<1$. We have $$\mbbe\bem\|\mbf x^k-\mbf x_\star^0\|^2\eem\leq(1+\ve)^k\eta^k(\|\mbf x^0-\mbf x_\star^0\|^2+(1+\ve)\alpha_{\rm r}^2\lambda_{\max}^{\rm r}\|\mbf z^0-(\mbf I-\mbf A\mbf A^\dag)\mbf b\|^2/\ve^2),$$ which shows that the {\rm EBRSI} algorithm converges  linearly  in the mean square sense to $\mbf x_\star^0$ with the rate $(1+\ve)\eta$. 
\end{remark}

\subsection{The randomized extended Kaczmarz algorithm} The REK algorithm \cite{zouzias2013rando} is one special case of the EBRSI algorithm. Choosing $\dsp\mbf S=\frac{\|\mbf A\|_\rmf}{\|\mbf A_{i,:}\|}\mbf I_{:,i}$ with probability $\dsp\frac{\|\mbf A_{i,:}\|^2}{\|\mbf A\|_\rmf^2}$ and $\dsp\mbf T=\frac{\|\mbf A\|_\rmf}{\|\mbf A_{:,j}\|}\mbf I_{:,j}$ with probability $\dsp\frac{\|\mbf A_{:,j}\|^2}{\|\mbf A\|_\rmf^2}$, we have $$\mbbe\bem\mbf S\mbf S^\top\eem=\mbf I,\quad \mbbe\bem\mbf T\mbf T^\top\eem=\mbf I,$$ and obtain 
\begin{align*}
\mbf z^k & =\mbf z^{k-1}-\alpha_{\rm c}\frac{\|\mbf A\|_\rmf^2}{\|\mbf A_{:,j}\|^2}(\mbf A_{:,j})^\top\mbf z^{k-1}\mbf A_{:,j},\\
\mbf x^k & =\mbf x^{k-1}-\alpha_{\rm r}\frac{\|\mbf A\|_\rmf^2}{\|\mbf A_{i,:}\|^2}(\mbf A_{i,:}\mbf x^{k-1}-\mbf b_i+\mbf z^k_i)(\mbf A_{i,:})^\top.
	\end{align*} We have $\lambda_{\max}^{\rm r}=\lambda_{\max}^{\rm c}=\|\mbf A\|_\rmf^2$. Setting $\alpha_{\rm r}=\alpha_{\rm c}=1/\|\mbf A\|_\rmf^2$, we recover the REK algorithm \cite{zouzias2013rando,du2019tight}. 

\subsection{The randomized extended average block Kaczmarz algorithm} The randomized extended average block Kaczmarz (REABK) algorithm \cite{du2020rando} is one special case of the EBRSI algorithm. Let $\{\mcali_1,\mcali_2,\ldots,\mcali_s\}$ be a partition of $[m]$ satisfying $\mcali_i\cap\mcali_j=\emptyset$ for $i\neq j$ and $\cup_{i=1}^s\mcali_i=[m]$. Let $\{\mcalj_1,\mcalj_2,\ldots,\mcalj_t\}$ be a partition of $[n]$ satisfying $\mcalj_i\cap\mcalj_j=\emptyset$ for $i\neq j$ and $\cup_{j=1}^t\mcalj_j=[n]$.  Choosing $\dsp\mbf S=\frac{\|\mbf A\|_\rmf}{\|\mbf A_{\mcali_i,:}\|_\rmf}\mbf I_{:,\mcali_i}$ with probability $\dsp\frac{\|\mbf A_{\mcali_i,:}\|_\rmf^2}{\|\mbf A\|_\rmf^2}$ and $\dsp\mbf T=\frac{\|\mbf A\|_\rmf}{\|\mbf A_{:,\mcalj_j}\|_\rmf}\mbf I_{:,\mcalj_j}$ with probability $\dsp\frac{\|\mbf A_{:,\mcalj_j}\|_\rmf^2}{\|\mbf A\|_\rmf^2}$, we have $$\mbbe\bem\mbf S\mbf S^\top\eem=\mbf I,\quad \mbbe\bem\mbf T\mbf T^\top\eem=\mbf I,$$ and obtain \begin{align*} \mbf z^k & =\mbf z^{k-1}-\alpha_{\rm c}\frac{\|\mbf A\|_\rmf^2}{\|\mbf A_{:,\mcalj_j}\|_\rmf^2}\mbf A_{:,\mcalj_j}(\mbf A_{:,\mcalj_j})^\top\mbf z^{k-1},\\ \mbf x^k & =\mbf x^{k-1}-\alpha_{\rm r}\frac{\|\mbf A\|_\rmf^2}{\|\mbf A_{\mcali_i,:}\|_\rmf^2}(\mbf A_{\mcali_i,:})^\top(\mbf A_{\mcali_i,:}\mbf x^{k-1}-\mbf b_{\mcali_i}+\mbf z^k_{\mcali_i}).\end{align*} Setting $\alpha_{\rm r}=\alpha_{\rm c}=\alpha/\|\mbf A\|_\rmf^2$, we recover the REABK algorithm \cite{du2020rando}.

\subsection{The extended block row uniform sampling algorithm} We propose one new special case of the EBRSI algorithm by using uniform sampling and refer to it as the extended block row uniform sampling (EBRUS) algorithm. Assume $1\leq\ell\leq \min\{m,n\}$. Let $\mcali$ (resp. $\mcalj$) denote the set consisting of the uniform sampling of $\ell$ different numbers of $[m]$  (resp. $[n]$). Setting $\dsp\mbf S=\sqrt{{m}/{\ell}}\mbf I_{:,\mcali}$ and $\mbf T=\sqrt{{n}/{\ell}}\mbf I_{:,\mcalj}$, we have $$\mbbe\bem\mbf S\mbf S^\top\eem=\frac{\dsp\frac{m}{\ell}}{\begin{pmatrix}m\\ \ell\end{pmatrix}}\sum_{\mcali\subseteq[m],\ |\mcali|=\ell}\mbf I_{:,\mcali}\mbf I_{:,\mcali}^\top=\mbf I,\quad \mbbe\bem\mbf T\mbf T^\top\eem=\frac{\dsp\frac{n}{\ell}}{\begin{pmatrix}n\\ \ell\end{pmatrix}}\sum_{\mcalj\subseteq[n],\ |\mcalj|=\ell}\mbf I_{:,\mcalj}\mbf I_{:,\mcalj}^\top=\mbf I,$$ and obtain \begin{align*}\mbf z^k &=\mbf z^{k-1}-\alpha_{\rm c}\frac{n}{\ell}\mbf A_{:,\mcalj}{(\mbf A_{:,\mcalj})^\top\mbf z^{k-1}},\\ \mbf x^k&=\mbf x^{k-1}-\alpha_{\rm r}\frac{m}{\ell} (\mbf A_{\mcali,:})^\top({\mbf A_{\mcali,:}\mbf x^{k-1}-\mbf b_\mcali+\mbf z_\mcali^k}). \end{align*} We have $$\lambda_{\max}^{\rm r}=\frac{m}{\ell}\max_{\mcali\subseteq[m],|\mcali|=\ell}\|\mbf A_{\mcali,:}\|^2, \quad\mbox{and}\quad \lambda_{\max}^{\rm c}=\frac{n}{\ell}\max_{\mcalj\subseteq[n], |\mcalj|=\ell}\|\mbf A_{:,\mcalj}\|^2.$$ By Theorem \ref{sel}, the EBRUS algorithm can have a faster convergence rate than that of the REK algorithm if there exists $1\leq\ell\leq\min\{m,n\}$ satisfying $$\frac{m}{\ell}\max_{\mcali\subseteq[m],|\mcali|=\ell}\|\mbf A_{\mcali,:}\|^2\leq\|\mbf A\|_\rmf^2,\quad\mbox{and}\quad \frac{n}{\ell}\max_{\mcalj\subseteq[n], |\mcalj|=\ell}\|\mbf A_{:,\mcalj}\|^2\leq\|\mbf A\|_\rmf^2.$$
	
We present the details of the EBRUS algorithm with block size $\ell$ and initial guesses $\mbf x^0=\mbf 0$ and $\mbf z^0=\mbf b$ in Algorithm 3. We also note that the constants $m/\ell$ and $n/\ell$ are incorporated in the stepsize parameters $\alpha_{\rm r}$ and $\alpha_{\rm c}$, respectively.

\begin{center}
\begin{tabular*}{160mm}{l}
\toprule {\bf Algorithm 3:} EBRUS($\ell$)\\ 
\hline \noalign{\smallskip}
\qquad Initialize $\mbf z^0=\mbf b$, $\mbf x^0=\mbf 0$ and a fixed $1\leq\ell\leq\min\{m,n\}$\\
\qquad {\bf for} $k=1,2,\ldots $ {\bf do}\\
\qquad \qquad  Select randomly a set $\mcalj$ consisting of the uniform sampling of $\ell$ numbers of $[n]$\\
\qquad \qquad  Update $\mbf z^k=\mbf z^{k-1}-\alpha_{\rm c}  \mbf A_{:,\mcalj}{(\mbf A_{:,\mcalj})^\top\mbf z^{k-1}}$\\
\qquad \qquad  Select randomly a set $\mcali$ consisting of the uniform sampling of $\ell$ numbers of $[m]$\\
\qquad \qquad  Update $\mbf x^k=\mbf x^{k-1}-\alpha_{\rm r} (\mbf A_{\mcali,:})^\top({\mbf A_{\mcali,:}\mbf x^{k-1}-\mbf b_\mcali+\mbf z_\mcali^k}) $\\
\qquad {\bf end for} \\ \bottomrule
\end{tabular*}
\end{center}
	
\section{Numerical results} 
In this section, we report numerical results showing the influence of different stepsize parameters and different block sizes on the convergence of the BRUS, BCUS, and EBRUS algorithms. We also compare the performance of these three algorithms with the following ten randomized algorithms: the randomized Kaczmarz (RK) algorithm \cite{strohmer2009rando}, the greedy randomized Kaczmarz (GRK) algorithm \cite{bai2018greed}, the randomized block Kaczmarz (RBK) algorithm \cite{gower2015rando}, the randomized coordinate descent (RCD) algorithm \cite{leventhal2010rando}, the greedy randomized coordinate descent (GRCD) algorithm \cite{bai2019greed}, the randomized block coordinate descent (RBCD) algorithm \cite{gower2015rando}, the randomized extended Kaczmarz (REK) algorithm \cite{zouzias2013rando}, the two-subspace randomized extended Kaczmarz (TREK) algorithm \cite{wu2021two}, the randomized extended block Kaczmarz (REBK) algorithm (which is slightly different from the randomized double block Kaczmarz algorithm of \cite{needell2015rando}), and the randomized extended average block Kaczmarz (REABK) algorithm \cite{du2020rando}. According to our theoretical results (Theorems \ref{rslt}, \ref{cslt}, and \ref{sel}), and taking into consideration the computational cost of each step, we divide the thirteen algorithms into three groups for comparison: (1) RK, GRK, RBK, and BRUS for solving arbitrary consistent linear systems; (2) RCD, GRCD, RBCD, and BCUS for solving full column rank (i.e., $\rank(\mbf A)=n$) inconsistent linear systems; (3) REK, TREK, REBK, REABK, and EBRUS for solving rank-deficient (i.e., $\rank(\mbf A)<n$) inconsistent linear systems. 

All experiments are performed using MATLAB R2019a on an iMac with 3.4 GHz Intel Core i5 processor and 8 GB 1600 MHz DDR3 memory. Given a matrix $\mbf A\in\mbbr^{m\times n}$, a consistent system is constructed by setting {\tt b=A*randn(n,1)}, and an inconsistent one is constructed by setting {\tt b=A*randn(n,1)+null(A}'{\tt)*randn(m-r,1)}, where {\tt r} is the rank of {\tt A}. Both synthetic data matrix and real-world data matrix are tested. All synthetic data matrices  are generated as follows. Given $m$, $n$, $r = \rank(\mbf A)\leq\min\{m,n\}$, and $\kappa\geq 1$, we construct $\mbf A$ by $\bf A = UDV^\top$, where $\mbf U\in\mbbr^{m\times r}$, $\mbf D\in\mbbr^{r\times r}$ and $\mbf V\in\mbbr^{n\times r}$ are given by {\tt [U,$\sim$]=qr(randn(m,r),0)}, {\tt D=diag(ones(r,1)+($\kappa$-1)*rand(r,1))}, and {\tt [V,$\sim$]=qr(randn(n,r),0)}. So the condition number of $\mbf A$, $\sigma_{\rm max}(\mbf A)/\sigma_{\rm min}(\mbf A)$, is bounded by $\kappa$. We use $\kappa=5$ for all  synthetic data matrices.

We note that {\tt A$\backslash$b} will be the same as {\tt pinv(A)*b} when {\tt A} have full column rank and usually not be the same as {\tt pinv(A)*b} when {\tt A} is rank-deficient. Therefore, we use MATLAB's `$\backslash$' to solve the small least squares problem at each step of the RBCD algorithm, and use MATLAB's {\tt lsqminnorm} (which is typically more efficient than {\tt pinv}) to solve the small least squares problems at each step of the RBK and REBK algorithms. For all algorithms, we use $\mbf x^0 = \mbf 0$ (and $\mbf z^0=\mbf b$ if needed) and stop if the relative error ({\tt relerr}), defined by $${\tt relerr}=\frac{\|\mbf x^k-\mbf A^\dag\mbf b\|^2}{\|\mbf A^\dag\mbf b\|^2},$$ satisfies ${\tt relerr}\leq 10^{-10}$. We check this stopping criterion after each epoch. For the RK, GRK, RBK, and BRUS  algorithms, an epoch consists of $m$, $m$, $\lceil m/\ell\rceil$, and $\lceil m/\ell\rceil$ iterations, respectively. For the RCD, GRCD, RBCD, and BCUS algorithms, an epoch consists of $n$, $n$, $\lceil n/\ell\rceil$, and $\lceil n/\ell\rceil$ iterations, respectively. For the REK, TREK, REBK, REABK, and EBRUS algorithms, an epoch consists of $\max\{m,n\}$, $\lceil\max\{m,n\}/2\rceil$, $\lceil\max\{m,n\}/\ell\rceil$, and $\lceil\max\{m,n\}/\ell\rceil$ iterations, respectively. In all examples, the reported results are average of 10 independent trials. 

\begin{figure}[tbhp]
  \centering
  \subfloat[BRUS]{\includegraphics[height=4.4cm]{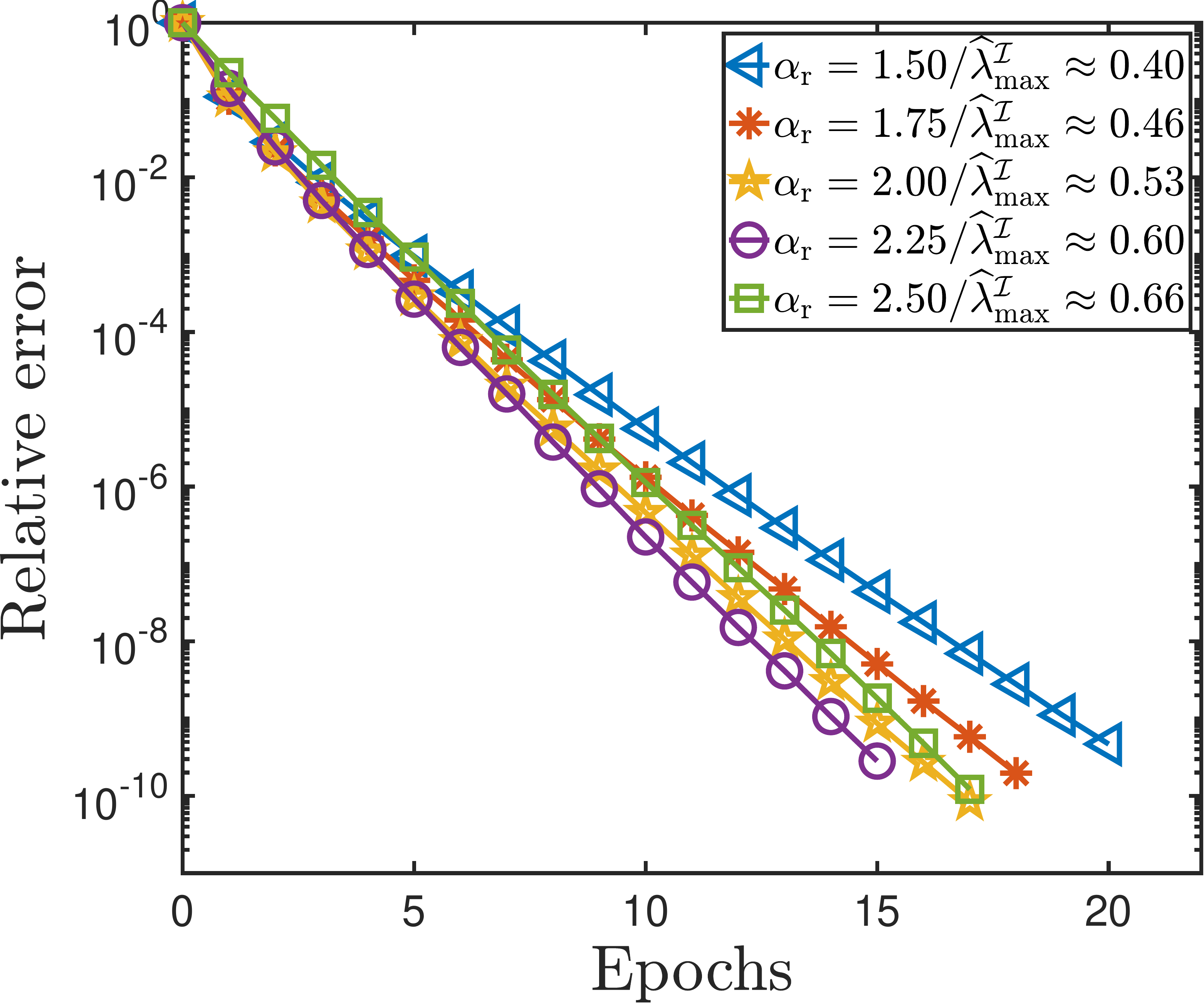}}\quad
  \subfloat[BCUS]{\includegraphics[height=4.4cm]{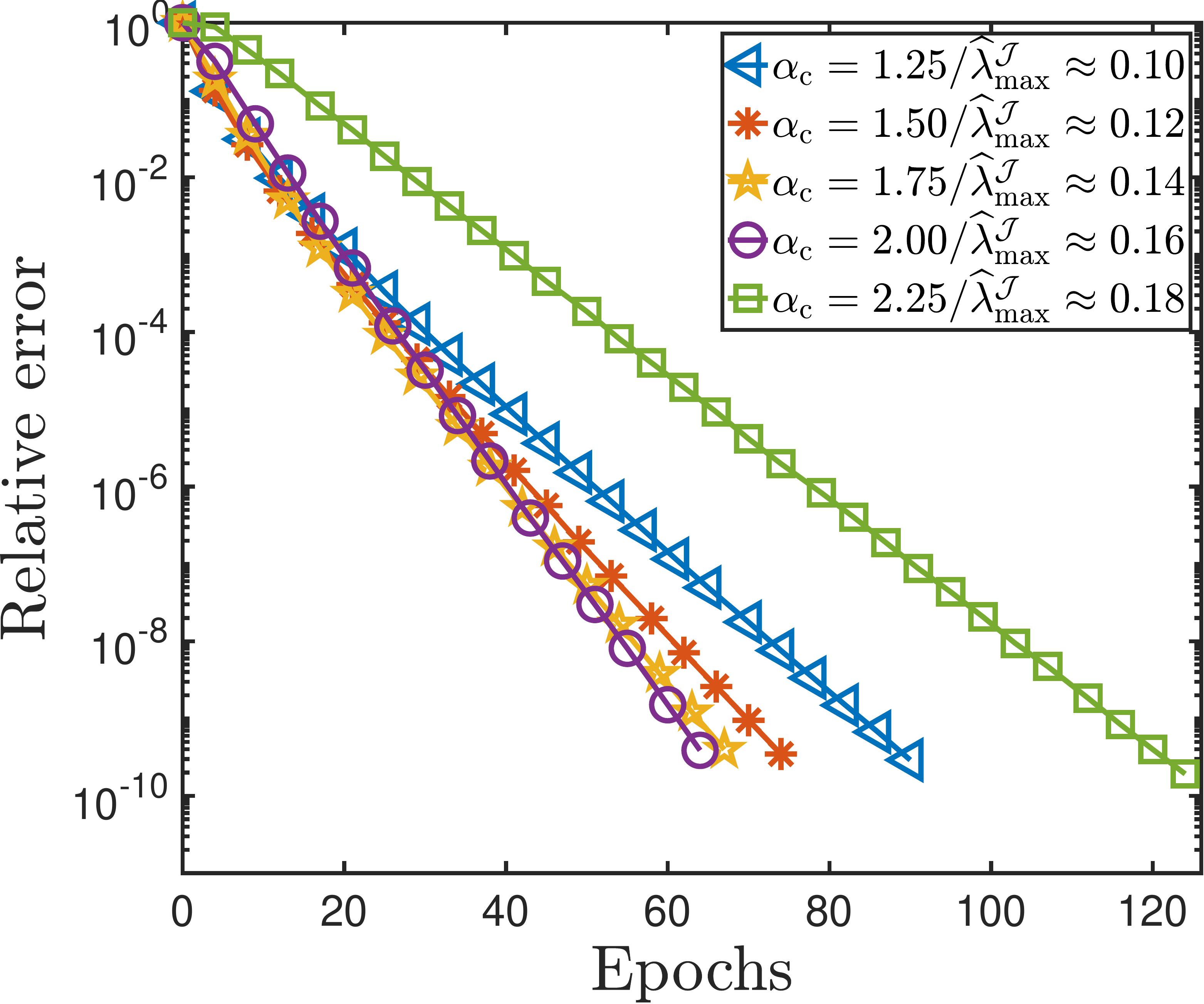}}\quad
  \subfloat[EBRUS]{\includegraphics[height=4.4cm]{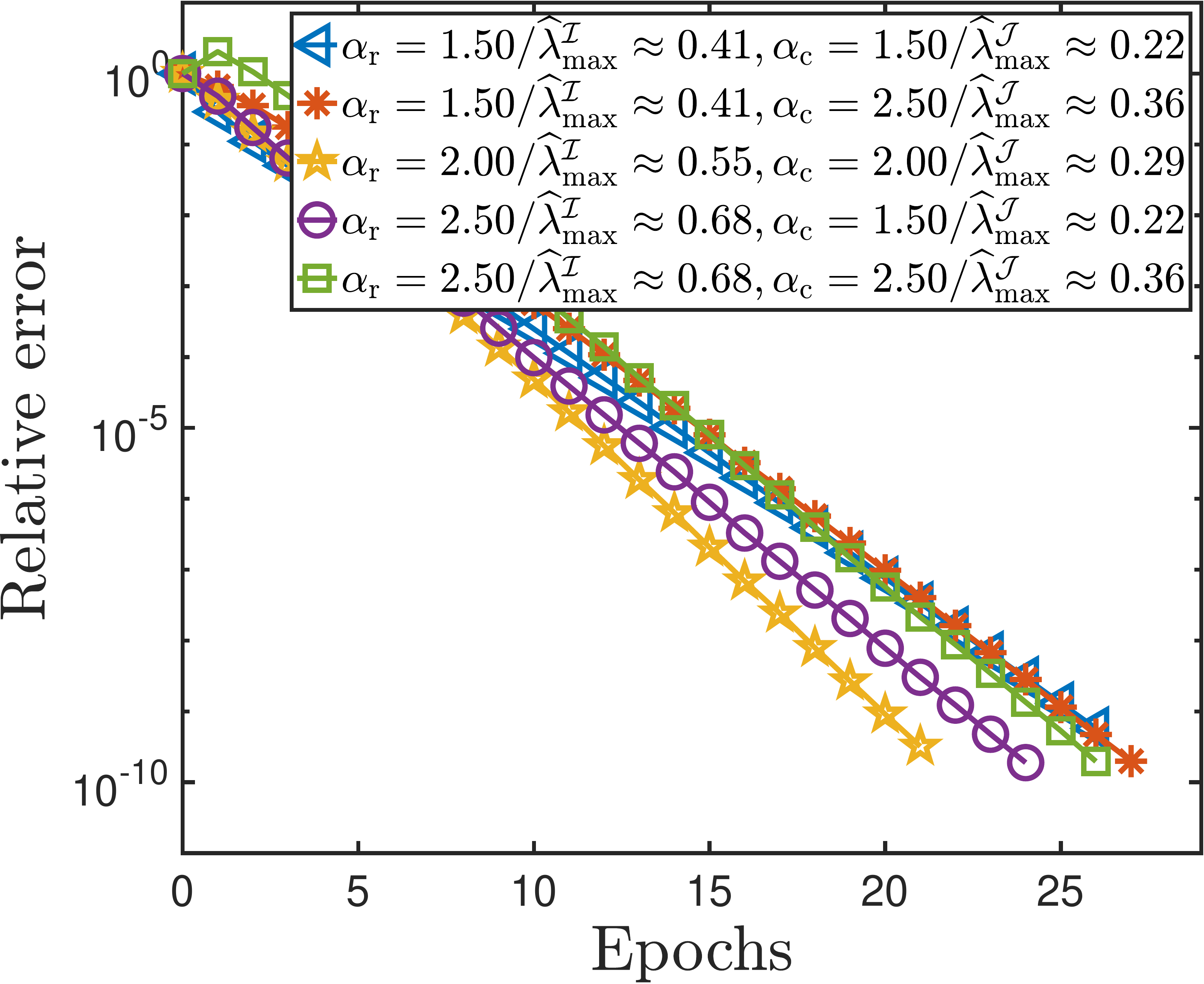}}
  \caption{The relative error history versus the number of epochs of the BRUS, BCUS, and EBRUS algorithms with different stepsize parameters. (a) BRUS for a consistent linear system with synthetic data matrix ($m=1000, n=500, r=250$). (b) BCUS for a full column rank inconsistent linear system with synthetic data matrix ($m=1000, n=500, r=500$). (c) EBRUS for a rank-deficient inconsistent linear system with synthetic data matrix ($m=1000, n=500, r=250$).}
  \label{fig:cor}
\end{figure}

First we use three linear systems to investigate how different stepsize parameters affect the convergence of the BRUS, BCUS, and EBRUS algorithms. We use the block size $\ell=10$ and define $$\wh\lambda_{\max}^{\mcali}:=\max_{1\leq i\leq  \ell}\|\mbf A_{\mcali_i,:}\|^2, \quad \wh\lambda_{\max}^{\mcalj}:=\max_{1\leq i\leq \ell}\|\mbf A_{:,\mcalj_i}\|^2,$$ where $\{\mcali_i\}_{i=1}^\ell$ are independent sets consisting of the uniform sampling of $\ell$ different numbers of $[m]$, and $\{\mcalj_i\}_{i=1}^\ell$ are independent sets consisting of the uniform sampling of $\ell$ different numbers of $[n]$.  In Figure \ref{fig:cor}, we plot the relative error history versus the number of epochs of the BRUS, BCUS, and EBRUS algorithms with different stepsize parameters for solving a consistent linear system with synthetic data matrix ($m=1000, n=500, r=250$), a full column rank inconsistent linear system with synthetic data matrix ($m=1000, n=500, r=500$), and a rank-deficient inconsistent linear system with synthetic data matrix ($m=1000, n=500, r=250$), respectively. Because the required number of epochs of each independent trial is different, the average relative error history is only available up to the minimum number of epochs. For the BRUS and BCUS algorithms, we observe that the convergence rate becomes faster as the increase of the stepsize, and then slows down after reaching the fastest rate. For the EBRUS algorithm, we observe that appropriate stepsize parameters can remarkably improve the convergence.

\begin{figure}[tbhp]
  \centering
  \subfloat[BRUS]{\includegraphics[height=4.4cm]{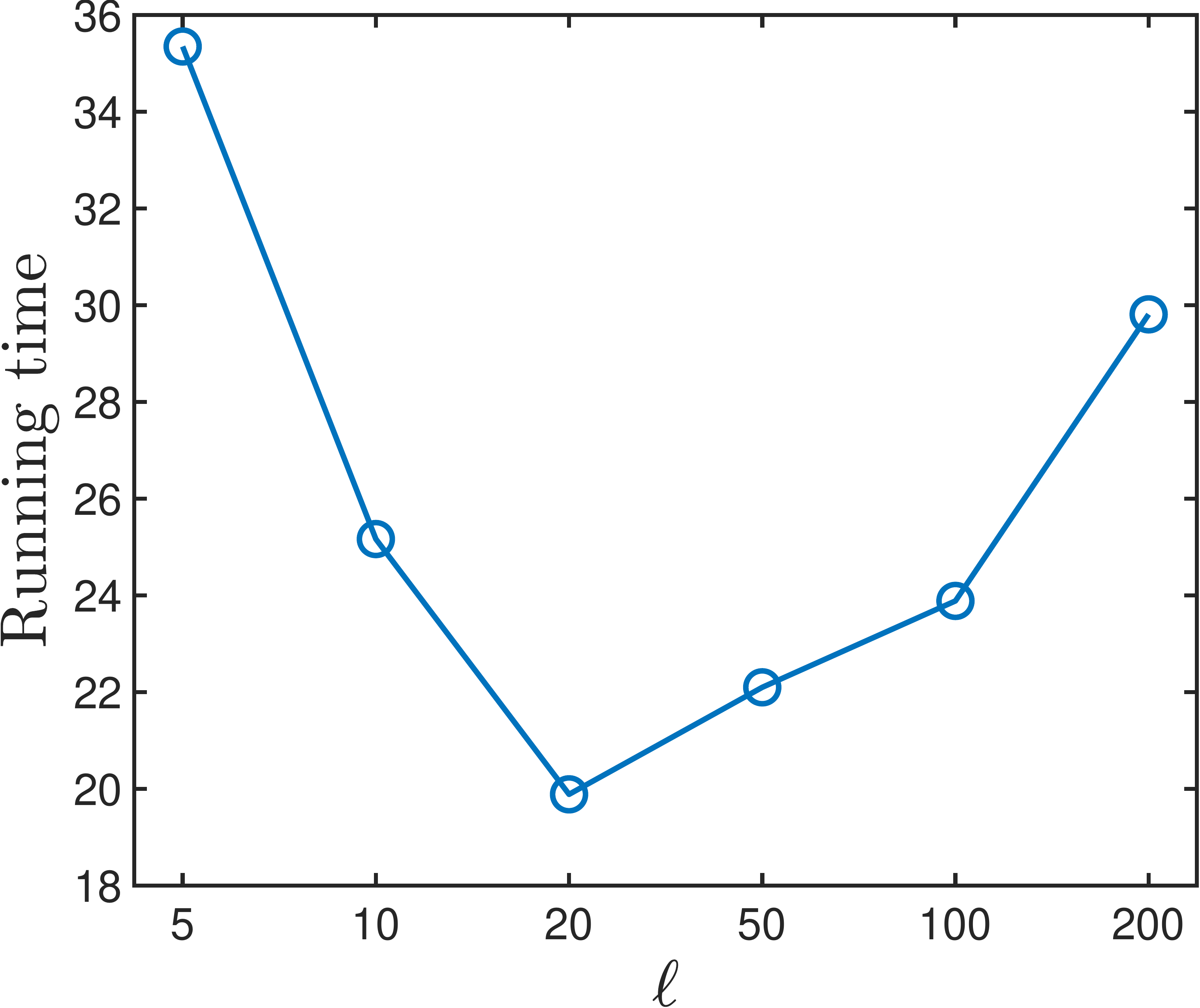}}\quad
  \subfloat[BCUS]{\includegraphics[height=4.4cm]{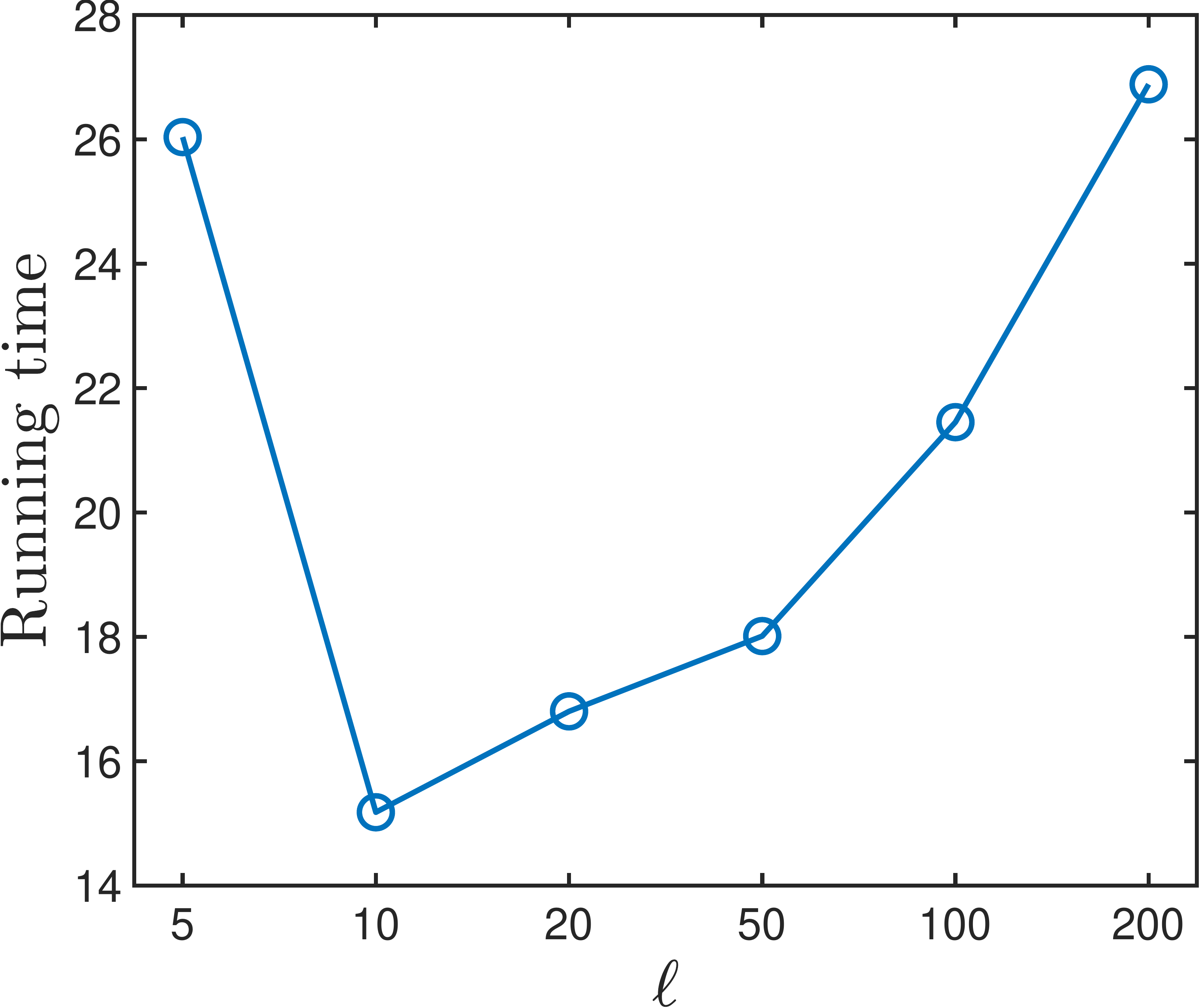}}\quad
  \subfloat[EBRUS]{\includegraphics[height=4.4cm]{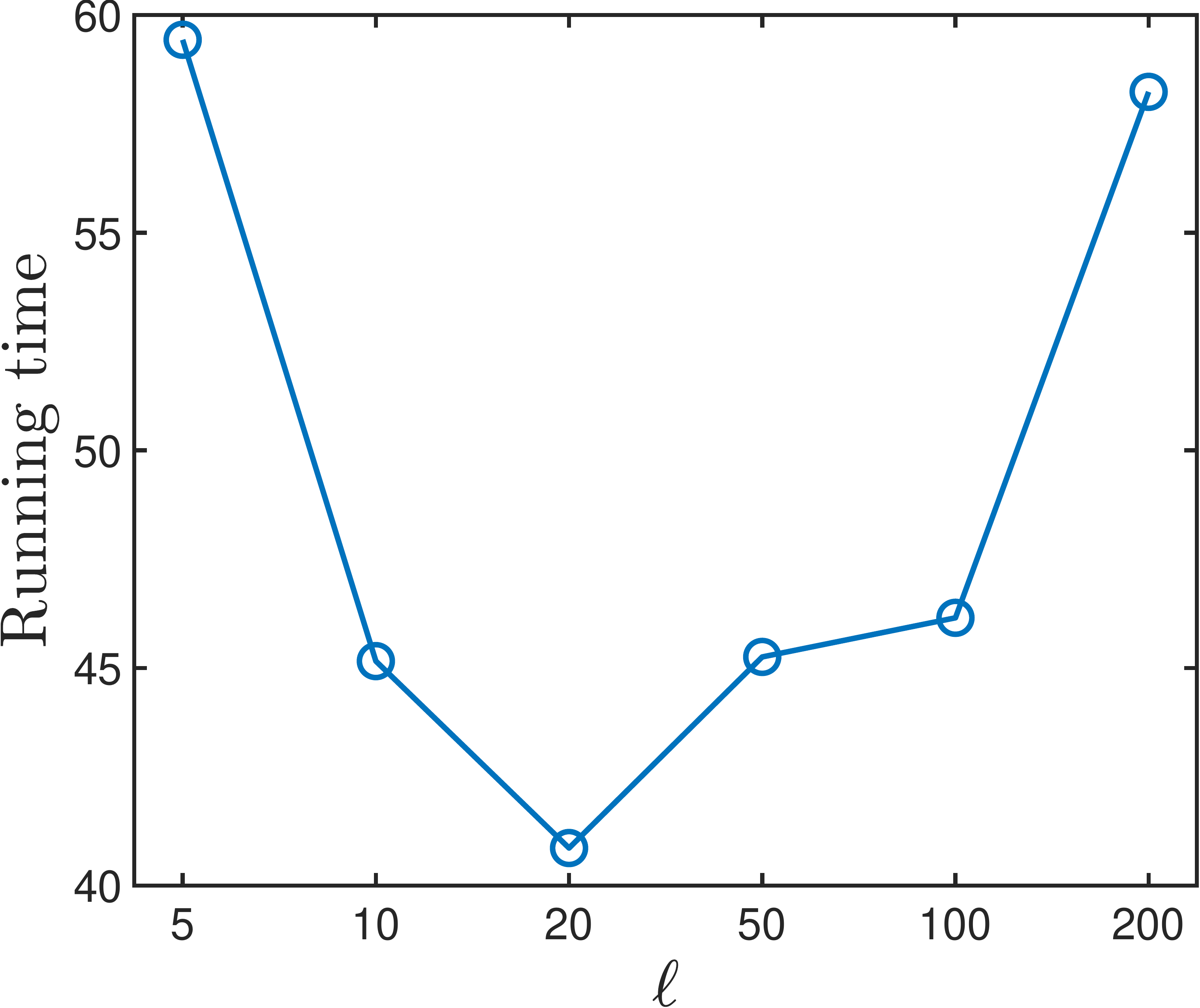}}
  \caption{The running time versus the block size ($\ell=5,10,20,50,100,200$) of the BRUS, BCUS, and EBRUS algorithms with stepsize parameters $\alpha_{\rm r}=2/\wh\lambda_{\max}^{\mcali}$ and  $\alpha_{\rm c}=2/\wh\lambda_{\max}^{\mcalj}$.  (a) BRUS for a consistent linear system with synthetic data matrix ($m=10000, n=5000, r=2500$). (b) BCUS for a full column rank inconsistent linear system with synthetic data matrix ($m=10000, n=5000, r=5000$). (c) EBRUS for a rank-deficient inconsistent linear system with synthetic data matrix ($m=10000, n=5000, r=2500$).}
  \label{fig:ell}
\end{figure}

Next we use three linear systems to investigate how different block sizes affect the performance of the BRUS, BCUS, and EBRUS algorithms. We use the block sizes $\ell=5,10,20,50,100,200$ and the empirical stepsize parameters $\alpha_{\rm r}=2/\wh\lambda_{\max}^{\mcali}$ and  $\alpha_{\rm c}=2/\wh\lambda_{\max}^{\mcalj}$. In Figure \ref{fig:ell}, we plot the running time (for all the reported results in this section, the time for computing stepsize parameters $\alpha_{\rm r}$ and $\alpha_{\rm c}$ is contained) versus the block size of the BRUS, BCUS, and EBRUS algorithms for solving a consistent linear system with synthetic data matrix ($m=10000, n=5000, r=2500$),  a full column rank inconsistent linear system with synthetic data matrix ($m=10000, n=5000, r=5000$), and a rank-deficient inconsistent linear system with synthetic data matrix ($m=10000, n=5000, r=2500$), respectively. For the BRUS, BCUS, and EBRUS algorithms, we observe that the running time first decreases, and then increases after reaching the minimum value with the increase of block size. 

\begin{table}[htbp]
{
  \caption{The number of epochs ({\tt epochs}), the number of iterations ({\tt iters}), the relative error ({\tt relerr}), and the running time ({\tt runtime}) of the RK, GRK, RBK($\ell$), and BRUS($\ell$) algorithms for seven consistent linear systems (both full column rank and rank-deficient cases are included).  Here $\ell$ is the block size.}  \label{tab:cor}
\begin{center}
  \begin{tabular}{|c|c|c|c|c|c|c|c|c|} \hline
   matrix &$m$ & $n$ & rank & algorithm  &  {\tt epochs} &{\tt iters} &{\tt relerr} & {\tt runtime} \\ \hline
    &&&& RK & 51.2 & 25600 & 8.78E$-11$ & 1.89\\     
    \multirow{2}{*}{\tt synth\_urd} & \multirow{2}{*}{500} & \multirow{2}{*}{2000} & \multirow{2}{*}{250} & GRK & {\bf 12.4} & 6200 & 5.78E$-11$& 1.74\\ 
    &&&& RBK(20) & 45.4 & 1135 & 8.30E$-11$& 0.96\\ 
    &&&& BRUS(20) & 42.4 & 1060 & 8.34E$-11$ & {\bf 0.42}\\ \hline
     &&&& RK & 12.0 & 24000 & 4.62E$-11$ & 1.29\\     
    \multirow{2}{*}{\tt synth\_ord} & \multirow{2}{*}{2000} & \multirow{2}{*}{500} & \multirow{2}{*}{250} & GRK & {\bf 2.0} & 4000 & 1.55E$-11$ & 1.10\\ 
    &&&& RBK(20) & 10.6 & 1060 & 5.27E$-11$ & 0.32\\
    &&&& BRUS(20) & 11.2 & 1120 & 4.41E$-11$ & {\bf 0.13}\\ \hline
    &&&& RK & 22.7 & 45400 & 6.71E$-11$ & 2.44\\     
    \multirow{2}{*}{\tt synth\_ofr} & \multirow{2}{*}{2000} & \multirow{2}{*}{500} & \multirow{2}{*}{500} & GRK & {\bf 5.0} & 10000 & 4.06E$-12$& 2.72\\
    &&&& RBK(20) & 21.6 & 2160 & 6.11E$-11$ & 0.63 \\
    &&&& BRUS(20) & 17.8 & 1780 & 5.28E$-11$ & {\bf 0.19}\\ \hline
        &&&& RK  & 12.5 & 182450 & 4.89E$-11$ & 39.68\\ 
  \multirow{2}{*}{\tt abtaha1}  & \multirow{2}{*}{14596} & \multirow{2}{*}{209} & \multirow{2}{*}{209} & GRK & {\bf 1.0} & 14596 & 4.77E$-31$ & 3.19\\ 
    &&&& RBK(20) & 6.5 & 4745 & 4.34E$-11$ & {\bf 2.61} \\
    &&&& BRUS(20) & 13.2 & 9636 & 6.79E$-11$ & 2.88 \\ \hline
    &&&& RK & 11.3 & 10825 & 4.44E$-11$ & 0.48\\ 
  \multirow{2}{*}{\tt ash958}  & \multirow{2}{*}{958}  & \multirow{2}{*}{292} & \multirow{2}{*}{292}  & GRK & {\bf 2.0} & 1916 & 1.17E$-14$ & 0.09\\
    &&&& RBK(10) & 11.0 & 1056 & 3.35E$-11$ & 0.25\\
    &&&& BRUS(10) & 11.1 & 1066 & 4.28E$-11$ & {\bf 0.02}\\ \hline
    &&&& RK &22.0 & 139260 & 5.51E$-11$ & 83.12\\     
    \multirow{2}{*}{\tt lp\_nug15}  & \multirow{2}{*}{6330} & \multirow{2}{*}{22275} & \multirow{2}{*}{5698} & GRK & {\bf 5.0} & 31650 & 1.16E$-12$ & 16.58\\ 
    &&&& RBK(20) & 21.7 & 6879 & 5.77E$-11$ & 8.89\\ 
    &&&& BRUS(20) & 46.6 & 14772 & 7.90E$-11$ & {\bf 8.44}\\ \hline
    &&&& RK & 15.1 & 331052 & 4.09E$-11$ & 117.69\\
    \multirow{2}{*}{\tt relat7}  & \multirow{2}{*}{21924} & \multirow{2}{*}{1045} & \multirow{2}{*}{1012} & GRK & {\bf 2.0} & 43848 & 4.02E$-11$ & 16.67\\
    &&&& RBK(20) & 15.6 & 17113 & 5.78E$-11$ & 10.51\\
    &&&& BRUS(20) & 14.9 & 16345 & 5.24E$-11$ & {\bf 4.86}\\ \hline
  \end{tabular}
\end{center}
}
\end{table} 

\begin{table}[htbp]
{
  \caption{The number of epochs ({\tt epochs}), the number of iterations ({\tt iters}), the relative error ({\tt relerr}), and the running time ({\tt runtime}) of the RCD, GRCD, RBCD($\ell$), and BCUS($\ell$) algorithms for three full column rank inconsistent linear systems.  Here $\ell$ is the block size.}  \label{tab:inf}
\begin{center}
  \begin{tabular}{|c|c|c|c|c|c|c|c|c|} \hline
   matrix & $m$ & $n$ & rank & algorithm & {\tt epochs} & {\tt iters}& {\tt relerr} & {\tt runtime} \\ \hline
    &&&& RCD & 97.8 & 48900 & 8.68E$-11$ & 1.36\\ 
 \multirow{2}{*}{\tt synth\_ofr} & \multirow{2}{*}{2000} & \multirow{2}{*}{500} & \multirow{2}{*}{500} & GRCD & {\bf 29.6} & 14800 & 7.47E$-11$ & 3.59\\ 
    &&&& RBCD(20) & 90.7 & {\bf 2268} & 8.83E$-11$ & 0.81\\ 
    &&&& BCUS(20) & 125.3 & 3133 & 9.05E$-11$ & {\bf 0.31}\\ \hline
    &&&& RCD  & 769.9 & 160909 & 9.79E$-11$ & 6.46\\ 
  \multirow{2}{*}{\tt abtaha1}  & \multirow{2}{*}{14596} & \multirow{2}{*}{209} & \multirow{2}{*}{209} & GRCD & {\bf 133.6} & 27922 & 9.42E$-11$ & 1.89\\ 
    &&&& RBCD(5) & 472.0 & {\bf 19824} & 9.81E$-11$ & 15.94 \\
    &&&& BCUS(5) & 1056.9 & 44390 & 9.93E$-11$ & {\bf 1.26} \\ \hline
    &&&& RCD & 33.2 & 9694 & 6.61E$-11$ & 0.26 \\ 
  \multirow{2}{*}{\tt ash958}  & \multirow{2}{*}{958}  & \multirow{2}{*}{292} & \multirow{2}{*}{292}  & GRCD & {\bf 6.0} & 1752 & 1.75E$-11$ & 0.06\\ 
    &&&& RBCD(5) & 23.4 & {\bf 1381} & 4.79E$-11$ & 0.10 \\
    &&&& BCUS(5) & 53.0 & 3127 & 8.48E$-11$ & {\bf 0.02}\\  \hline
  \end{tabular}
\end{center}
}
\end{table}

\begin{table}[htbp]
{
\caption{The number of epochs ({\tt epochs}), the number of iterations ({\tt iters}), the relative error ({\tt relerr}), and the running time ({\tt runtime}) of the REK, TREK, REBK($\ell$), REABK($\ell$), and EBRUS($\ell$) algorithms for four rank-deficient inconsistent linear systems. Here $\ell$ is the block size.}  \label{tab:inr}
\begin{center}
  \begin{tabular}{|c|c|c|c|c|c|c|c|c|} \hline
   matrix & $m$ & $n$ & rank& algorithm & {\tt epochs} & {\tt iters}& {\tt relerr} & {\tt runtime} \\ \hline
    & &  & & REK & 17.6 & 35200 & 6.28E$-11$ & 3.96 \\
 &&&& TREK  & 17.7 & 17700 & 4.17E$-11$ & 5.08  \\
    {\tt synth\_urd}   & {500} & {2000} & {250} &  REBK(20) & 15.6 & 1560 & 5.15E$-11$ & 1.61 \\     
    &&&&  REABK(20) & 18.4 & 1840 & 5.48E$-11$ & {\bf 0.57} \\
    &&&&  EBRUS(20) & 15.6 & 1560 & 4.29E$-11$ & 0.67 \\ \hline
 & &  & & REK & 16.9 & 33800 & 5.82E$-11$ & 2.83 \\
  &   &   &   & TREK & 16.7 & 16700 & 4.88E$-11$ & 3.22 \\
    {\tt synth\_ord}   & {2000} & {500} & {250}&  REBK(20) & 15.1 & 1510 & 5.09E$-11$ & 1.16\\
    &&&&  REABK(20) & 18.0 & 1800 & 5.88E$-11$ & 0.41 \\
    &&&&  EBRUS(20) & 15.2 & 1520 & 4.22E$-11$ & {\bf 0.31}\\ \hline
    &&&&  REK & 8.0 & 178200 & 8.69E$-12$ & 160.08 \\
   &&&&  TREK & 8.0 & 89104 & 1.34E$-11$ & 218.18 \\
    {\tt lp\_nug15}   &  {6330} &  {22275} &  {5698}&  REBK(20) & 13.1 & 14593 & 5.31E$-11$ & 26.18 \\    
    &&&&  REABK(20) & 20.8 & 23171 & 5.08E$-11$ & 18.20 \\
    &&&&  EBRUS(20) & 15.7 & 17490 & 5.38E$-11$ & {\bf 10.18} \\ \hline
    &&&&  REK & 16.9 & 370516 & 4.61E$-11$ & 154.30 \\
    &&&& TREK & 17.7 & 194027 & 5.83E$-11$ & 229.91 \\ 
    {\tt relat7}   & {21924} &  {1045} &  {1012} & REBK(20) & 15.0 & 16455 & 5.89E$-11$ & 35.36 \\
    &&&&  REABK(20) & 124.4 & 136467 & 9.26E$-11$ & 53.22 \\
    &&&& EBRUS(20) & 17.2 & 18868 & 5.74E$-11$ & {\bf 6.41} \\ \hline
  \end{tabular}
\end{center}
}
\end{table}

Last we compare the thirteen algorithms (in three groups) using three synthetic data matrices ({\tt synth\_urd}, {\tt synth\_ord}, and {\tt synth\_ofr})  and four real-world data matrices ({\tt abtaha1}, {\tt ash958}, {\tt lp\_nug15}, and {\tt relat7}) from the SuiteSparse Matrix Collection (formerly known as the University of Florida Sparse Matrix Collection) \cite{davis2011unive}. The four matrices, {\tt synth\_urd}, {\tt synth\_ord}, {\tt lp\_nug15}, and {\tt relat7},  are rank-deficient, and the other three matrices, {\tt synth\_ofr}, {\tt abtaha1}, and {\tt ash958}, have full column rank. 

In Table \ref{tab:cor}, we report the numerical results of the RK, GRK, RBK, and BRUS algorithms for seven consistent linear systems (both full column rank and rank-deficient cases are included). For the BRUS algorithm, we use the empirical stepsize parameter $\alpha_{\rm r}=2/\wh\lambda_{\max}^{\mcali}$. For all cases, the GRK algorithm is the best in terms of the number of epochs. For the case of {\tt abtaha1}, the RBK algorithm is the best in terms of the running time; and for the other six cases, the BRUS algorithm is the best.

In Table \ref{tab:inf}, we report the numerical results of the RCD, GRCD, RBCD, and BCUS algorithms for three full column rank inconsistent linear systems. For the BCUS algorithm, we use the empirical stepsize parameter $\alpha_{\rm c}=1/\wh\lambda_{\max}^{\mcalj}$. For all cases, the GRCD algorithm is the best in terms of the number of epochs, the RBCD algorithm is the best in terms of the number of iterations, and the BCUS algorithm is the best in terms of the running time. 

In Table \ref{tab:inr}, we report the numerical results of the REK, TREK, REBK, REABK, and EBRUS algorithms for four rank-deficient inconsistent linear systems. For the REABK algorithm, we use the same empirical stepsize parameter reported in \cite{du2020rando}. For the EBRUS algorithm, we use the empirical stepsize parameters $\alpha_{\rm r}=2/\widehat\lambda_{\max}^{\mcali}$ and $\alpha_{\rm c}=2/\wh\lambda_{\max}^{\mcalj}$. For the case of {\tt synth\_urd}, the REABK algorithm is the best in terms of the running time; and for the other three cases, the EBRUS algorithm is the best.
 
\section{Concluding remarks}
We have proposed two novel pseudoinverse-free randomized block iterative algorithms for solving consistent and inconsistent linear systems of equations. Our main results show that our algorithms converge linearly in the mean square sense to a (least squares) solution of the linear system $\mbf A\mbf x=\mbf b$. By using uniform sampling, we have designed the BRUS algorithm for solving consistent linear systems, the BCUS algorithm for solving full column rank inconsistent linear systems, and the EBRUS algorithm for solving rank-deficient inconsistent linear systems. Numerical experiments for both synthetic and real-world matrices demonstrate that the BRUS, BCUS, and EBRUS algorithms with appropriate stepsize parameters and block size can significantly outperform several existing randomized algorithms in terms of the running time.   

\section*{Acknowledgments}
The authors are thankful to the referees for their detailed comments and valuable suggestions that have led to remarkable improvements. This work was supported by the National Natural Science Foundation of China (No.12171403 and No.11771364), the Natural Science Foundation of Fujian Province of China (No.2020J01030), and the Fundamental Research Funds for the Central Universities (No.20720210032).
\appendix

\section{Extended block column and row sampling iterative algorithm}\label{appendix}

The RCD algorithm \cite{leventhal2010rando} for $\bf Az=b$ with arbitrary initial guess $\mbf z^0\in\mbbr^n$ produces a sequence $\{\mbf z^k\}$, which satisfies $\mbf A\mbf z^k\rightarrow\mbf A\mbf A^\dag\mbf b$. Ma, Needell, and Ramdas \cite{ma2015conve} proved that the $k$th iterate $\mbf x^k$ (which is produced by one RK update for ${\bf Ax=Az}^k$ from $\mbf x^{k-1}$) of the randomized extended Gauss--Seidel (REGS) algorithm \cite{ma2015conve} converges to a solution of $\bf Ax=AA^\dag b$. In this section, based on the idea of the REGS algorithm, we propose an extended block column and row sampling iterative (EBCRSI) algorithm.

Given arbitrary $\mbf z^0\in\mbbr^n$ and $\mbf x^0\in\mbbr^n$, the iterates of the EBCRSI algorithm at step $k$ are defined as 
\begin{align}
\mbf z^k &=\mbf z^{k-1}-\alpha_{\rm c}\mbf T\mbf T^\top\mbf A^\top(\mbf A\mbf z^{k-1}-\mbf b),\label{rgsz}\\
\mbf x^k & =\mbf x^{k-1}-\alpha_{\rm r}\mbf A^\top\mbf S\mbf S^\top\mbf A(\mbf x^{k-1}-\mbf z^k),\label{rgsx}
	\end{align}  where the random parameter matrices $\mbf S$ and $\mbf T$ are independent, and satisfy $$\mbbe\bem \mbf S\mbf S^\top \eem=\mbf I,\qquad \mbbe\bem \mbf T\mbf T^\top \eem=\mbf I.$$ We note that the iteration (\ref{rgsz}) is the BCSI algorithm for $\bf A z=b$ with arbitrary initial guess $\mbf z^0\in\mbbr^n$, and the iterate $\mbf x^k$ in (\ref{rgsx}) is one BRSI update for ${\bf Ax=Az}^k$ from $\mbf x^{k-1}$. By Theorem \ref{cslt}, we have \beq\label{sun2}\mbbe\bem \|\mbf A(\mbf z^k-\mbf A^\dag\mbf b)\|^2\eem\leq \eta_{\rm c}^k\|\mbf A(\mbf z^0-\mbf A^\dag\mbf b)\|^2,\eeq

In the following, we shall present two convergence results of the EBCRSI algorithms: Theorem \ref{rgs} is on the convergence of $\|\mbbe\bem \mbf x^k\eem-\mbf x_\star^0\|$, and Theorem \ref{nrgs} is on the convergence of $\mbbe\bem\|\mbf x^k-\mbf x_\star^0\|^2\eem$. We emphasize that both the convergence results hold for arbitrary linear systems.

\begin{theorem}\label{rgs}
 For arbitrary $\mbf z^0\in\mbbr^n$ and $\mbf x^0\in\mbbr^n$, the $k$th iterate $\mbf x^k$ of the {\rm EBCRSI} algorithm satisfies $$\mbbe\bem \mbf x^k-\mbf x_\star^0\eem=(\mbf I-\alpha_{\rm r}\mbf A^\top\mbf A)^k(\mbf x^0-\mbf x_\star^0)+\alpha_{\rm r}\sum_{i=0}^{k-1}(\mbf I-\alpha_{\rm r}\mbf A^\top\mbf A)^i(\mbf I-\alpha_{\rm c}\mbf A^\top\mbf A)^{k-i}\mbf A^\top(\mbf A\mbf z^0-\mbf b).$$ Moreover, \beq\label{ergs}
 	\|\mbbe\bem\mbf x^k\eem-\mbf x_\star^0\|\leq\delta^k(\|\mbf x^0-\mbf x_\star^0\|+k\alpha_{\rm r}\|\mbf A^\top(\mbf A\mbf z^0-\mbf b)\|)  
 \eeq
 where $$\delta=\max_{1\leq i\leq r}\{|1-\alpha_{\rm r}\sigma_i^2(\mbf A)|,|1-\alpha_{\rm c}\sigma_i^2(\mbf A)|\}.$$
\end{theorem}
	
\begin{proof} By $\mbf A^\top\mbf A\mbf A^\dag\mbf b=\mbf A^\top\mbf b$ and (\ref{rgsz}), we have
$$\mbf A\mbf z^k-\mbf A\mbf A^\dag \mbf b =\mbf A\mbf z^{k-1}-\mbf A\mbf A^\dag \mbf b-\alpha_{\rm c}\mbf A\mbf T\mbf T^\top\mbf A^\top(\mbf A\mbf z^{k-1}-\mbf A\mbf A^\dag \mbf b).$$ Taking conditional expectation conditioned on $\mbf z^{k-1}$ and $\mbf x^{k-1}$, we have
	\begin{align*}\mbbe_{k-1}\bem\mbf A\mbf z^k-\mbf A\mbf A^\dag \mbf b\eem &  =\mbf A\mbf z^{k-1}-\mbf A\mbf A^\dag \mbf b-\alpha_{\rm c}\mbf A\mbf A^\top(\mbf A\mbf z^{k-1}-\mbf A\mbf A^\dag \mbf b)\\ &=(\mbf I-\alpha_{\rm c}\mbf A\mbf A^\top)(\mbf A\mbf z^{k-1}-\mbf A\mbf A^\dag \mbf b).\end{align*} Then, by the law of total expectation, we have
\begin{align*}\mbbe\bem\mbf A\mbf z^k-\mbf A\mbf A^\dag \mbf b\eem &=(\mbf I-\alpha_{\rm c}\mbf A\mbf A^\top)\mbbe\bem\mbf A\mbf z^{k-1}-\mbf A\mbf A^\dag \mbf b\eem\\ &=(\mbf I-\alpha_{\rm c}\mbf A\mbf A^\top)^k(\mbf A\mbf z^0-\mbf A\mbf A^\dag \mbf b).\end{align*}
Taking expectation conditioned on $\mbf z^{k-1}$ and $\mbf x^{k-1}$ for $$\mbf x^k-\mbf x_\star^0 =\mbf x^{k-1}-\mbf x_\star^0-\alpha_{\rm r}\mbf A^\top\mbf S\mbf S^\top\mbf A(\mbf x^{k-1}-\mbf z^k),$$ we obtain
\begin{align*}\mbbe_{k-1}\bem\mbf x^k-\mbf x_\star^0\eem  
& =\mbbe_{k-1}\bem\mbbe_{k-1}^{\rm r}\bem \mbf x^{k-1}-\mbf x_\star^0 -\alpha_{\rm r}\mbf A^\top\mbf S\mbf S^\top\mbf A(\mbf x^{k-1}-\mbf z^k)\eem\eem\\
& =\mbbe_{k-1}\bem\mbf x^{k-1}-\mbf x_\star^0 -\alpha_{\rm r}\mbf A^\top\mbf A(\mbf x^{k-1}-\mbf x_\star^0+\mbf x_\star^0-\mbf z^k)\eem\\
& =(\mbf I-\alpha_{\rm r}\mbf A^\top\mbf A)(\mbf x^{k-1}-\mbf x_\star^0)+\alpha_{\rm r}\mbf A^\top\mbbe_{k-1}\bem\mbf A\mbf z^k-\mbf A\mbf x_\star^0\eem\\
& =(\mbf I-\alpha_{\rm r}\mbf A^\top\mbf A)(\mbf x^{k-1}-\mbf x_\star^0)+\alpha_{\rm r}\mbf A^\top\mbbe_{k-1}\bem\mbf A\mbf z^k-\mbf A\mbf A^\dag\mbf b\eem,
\end{align*} which, by the law of total expectation, yields
\begin{align*}
	\mbbe\bem\mbf x^k-\mbf x_\star^0\eem
	&=(\mbf I-\alpha_{\rm r}\mbf A^\top\mbf A)\mbbe\bem\mbf x^{k-1}-\mbf x_\star^0\eem +\alpha_{\rm r}\mbf A^\top\mbbe\bem\mbf A\mbf z^k-\mbf A\mbf A^\dag\mbf b\eem\\
	&=(\mbf I-\alpha_{\rm r}\mbf A^\top\mbf A)\mbbe\bem\mbf x^{k-1}-\mbf x_\star^0\eem +\alpha_{\rm r}\mbf A^\top(\mbf I-\alpha_{\rm c}\mbf A\mbf A^\top)^k(\mbf A\mbf z^0-\mbf A\mbf A^\dag \mbf b)\\
	&=(\mbf I-\alpha_{\rm r}\mbf A^\top\mbf A)\mbbe\bem\mbf x^{k-1}-\mbf x_\star^0\eem +\alpha_{\rm r}(\mbf I-\alpha_{\rm c}\mbf A^\top\mbf A)^k\mbf A^\top(\mbf A\mbf z^0-\mbf A\mbf A^\dag \mbf b)\\
	&=(\mbf I-\alpha_{\rm r}\mbf A^\top\mbf A)\mbbe\bem\mbf x^{k-1}-\mbf x_\star^0\eem +\alpha_{\rm r}(\mbf I-\alpha_{\rm c}\mbf A^\top\mbf A)^k\mbf A^\top(\mbf A\mbf z^0- \mbf b)\\
	&=(\mbf I-\alpha_{\rm r}\mbf A^\top\mbf A)^2\mbbe\bem\mbf x^{k-2}-\mbf x_\star^0\eem +\alpha_{\rm r}(\mbf I-\alpha_{\rm c}\mbf A^\top\mbf A)^k\mbf A^\top(\mbf A\mbf z^0-\mbf b)\\
	&\quad \ +\alpha_{\rm r}(\mbf I-\alpha_{\rm r}\mbf A^\top\mbf A)(\mbf I-\alpha_{\rm c}\mbf A^\top\mbf A)^{k-1}\mbf A^\top(\mbf A\mbf z^0- \mbf b)\\
	&=\cdots\\
	&=(\mbf I-\alpha_{\rm r}\mbf A^\top\mbf A)^k(\mbf x^0-\mbf x_\star^0)+\alpha_{\rm r}\sum_{i=0}^{k-1}(\mbf I-\alpha_{\rm r}\mbf A^\top\mbf A)^i(\mbf I-\alpha_{\rm c}\mbf A^\top\mbf A)^{k-i}\mbf A^\top(\mbf A\mbf z^0- \mbf b).
\end{align*} Taking 2-norm, by triangle inequality, $\mbf x^0-\mbf x_\star^0\in\ran(\mbf A^\top)$, $\mbf A^\top(\mbf A\mbf z^0- \mbf b)\in\ran(\mbf A^\top)$, and Lemma \ref{leqd}, we obtain the estimate (\ref{ergs}).
\end{proof}	
	
\begin{remark}
	In Theorem \ref{rgs}, no assumptions about the dimensions or rank of $\mbf A$ are assumed, and the system $\bf Ax=b$ can be consistent or inconsistent. If $0<\alpha_{\rm r}<2/\sigma_{\max}^2(\mbf A)$ and $0<\alpha_{\rm c}<2/\sigma_{\max}^2(\mbf A)$, then $0<\delta<1$. This means $\mbf x^k$ is an asymptotically unbiased estimator for $\mbf x_\star^0$.
\end{remark}	

\begin{theorem}\label{nrgs}
Assume that $0<\alpha_{\rm c}<2/\lambda_{\max}^{\rm c}$ and $0<\alpha_{\rm r}<2/\lambda_{\max}^{\rm r}$. For arbitrary $\mbf z^0\in\mbbr^n$, $\mbf x^0\in\mbbr^n$, and $\ve>0$, the $k$th iterate $\mbf x^k$ of the {\rm EBCRSI} algorithm satisfies 
\begin{align*}
\mbbe\bem\|\mbf x^k-\mbf x_\star^0\|^2\eem &\leq(1+\ve)^k\eta_{\rm r}^k\|\mbf x^0-\mbf x_\star^0\|^2\\ 
&\quad +(1+1/\ve)\alpha_{\rm r}^2\lambda_{\max}^{\rm r}\|\mbf A(\mbf z^0-\mbf A^\dag\mbf b)\|^2\sum_{i=0}^{k-1}\eta_{\rm c}^{k-i}(1+\ve)^i\eta_{\rm r}^i,
\end{align*}  where $$\eta_{\rm r}=1-\alpha_{\rm r}(2-\alpha_{\rm r}\lambda_{\max}^{\rm r})\sigma_{\min}^2(\mbf A),\quad \mbox{and}\quad  \eta_{\rm c}=1-\alpha_{\rm c}(2-\alpha_{\rm c}\lambda_{\max}^{\rm c})\sigma_{\min}^2(\mbf A).$$
\end{theorem}	

\begin{proof} By (\ref{rgsx}), we have $$\mbf x^k-\mbf x_\star^0=\mbf x^{k-1}-\mbf x_\star^0-\alpha_{\rm r}\mbf A^\top\mbf S\mbf S^\top\mbf A(\mbf x^{k-1}-\mbf x_\star^0+\mbf x_\star^0-\mbf z^k).$$ By triangle inequality and Young's inequality, we have 
	\begin{align*}\|\mbf x^k-\mbf x_\star^0\|^2
	&\leq (\|(\mbf I-\alpha_{\rm r}\mbf A^\top\mbf S\mbf S^\top\mbf A)(\mbf x^{k-1}-\mbf x_\star^0)\|+\alpha_{\rm r}\|\mbf A^\top\mbf S\mbf S^\top\mbf A(\mbf z^k-\mbf x_\star^0)\|)^2\\ 
	&\leq(1+\ve)\|(\mbf I-\alpha_{\rm r}\mbf A^\top\mbf S\mbf S^\top\mbf A)(\mbf x^{k-1}-\mbf x_\star^0)\|^2+(1+1/\ve)\alpha_{\rm r}^2\|\mbf A^\top\mbf S\mbf S^\top\mbf A(\mbf z^k-\mbf x_\star^0)\|^2.	
	\end{align*} 
By $\mbf x^0-\mbf x_\star^0\in\ran(\mbf A^\top)$ and $\mbf A^\top\mbf S\mbf S^\top\mbf A(\mbf x^{k-1}-\mbf z^k)\in\ran(\mbf A^\top)$, we can show that $\mbf x^k-\mbf x_\star^0\in\ran(\mbf A^\top)$ by induction. It follows that
\begin{align*}
	\|(\mbf I-\alpha_{\rm r}\mbf A^\top\mbf S\mbf S^\top\mbf A)(\mbf x^{k-1}-\mbf x_\star^0)\|^2 
	& = \|\mbf x^{k-1}-\mbf x_\star^0\|^2-2\alpha_{\rm r}(\mbf x^{k-1}-\mbf x_\star^0)^\top\mbf A^\top\mbf S\mbf S^\top\mbf A(\mbf x^{k-1}-\mbf x_\star^0)\\
	&\ \quad +\alpha_{\rm r}^2(\mbf x^{k-1}-\mbf x_\star^0)^\top(\mbf A^\top\mbf S\mbf S^\top\mbf A)^2(\mbf x^{k-1}-\mbf x_\star^0)\\
	&\leq \|\mbf x^{k-1}-\mbf x_\star^0\|^2-2\alpha_{\rm r}(\mbf x^{k-1}-\mbf x_\star^0)^\top\mbf A^\top\mbf S\mbf S^\top\mbf A(\mbf x^{k-1}-\mbf x_\star^0)\\
	&\ \quad +\alpha_{\rm r}^2\lambda_{\max}^{\rm r}(\mbf x^{k-1}-\mbf x_\star^0)^\top\mbf A^\top\mbf S\mbf S^\top\mbf A(\mbf x^{k-1}-\mbf x_\star^0).
\end{align*} Taking conditional expectation conditioned on $\mbf z^{k-1}$ and $\mbf x^{k-1}$, we obtain
$$\mbbe_{k-1}\bem\|(\mbf I-\alpha_{\rm r}\mbf A^\top\mbf S\mbf S^\top\mbf A)(\mbf x^{k-1}-\mbf x_\star^0)\|^2\eem\leq \eta_{\rm r}\|\mbf x^{k-1}-\mbf x_\star^0\|^2.$$  Taking conditional expectation conditioned on $\mbf z^{k-1}$ and $\mbf x^{k-1}$ for
$$\|\mbf A^\top\mbf S\mbf S^\top\mbf A(\mbf z^k-\mbf x_\star^0)\|^2\leq\lambda_{\max}^{\rm r}(\mbf z^k-\mbf x_\star^0)^\top\mbf A^\top\mbf S\mbf S^\top\mbf A(\mbf z^k-\mbf x_\star^0),$$ and by $\mbf A\mbf x_\star^0=\mbf A\mbf A^\dag\mbf b$, we obtain
\begin{align*}
\mbbe_{k-1}\bem\|\mbf A^\top\mbf S\mbf S^\top\mbf A(\mbf z^k-\mbf x_\star^0)\|^2\eem\leq \lambda_{\max}^{\rm r}\mbbe_{k-1}\bem\|\mbf A(\mbf z^k-\mbf x_\star^0)\|^2\eem=	\lambda_{\max}^{\rm r}\mbbe_{k-1}\bem\|\mbf A(\mbf z^k-\mbf A^\dag\mbf b)\|^2\eem.
\end{align*} It follows that 
$$\mbbe_{k-1}\bem\|\mbf x^k-\mbf x_\star^0\|^2\eem\leq(1+\ve)\eta_{\rm r}\|\mbf x^{k-1}-\mbf x_\star^0\|^2+(1+1/\ve)\alpha_{\rm r}^2\lambda_{\max}^{\rm r}\mbbe_{k-1}\bem\|\mbf A(\mbf z^k-\mbf A^\dag\mbf b)\|^2\eem.$$ Then, by the law of total expectation and the estimate (\ref{sun2}), we have
\begin{align*}
\mbbe\bem\|\mbf x^k-\mbf x_\star^0\|^2\eem
&\leq(1+1/\ve)\alpha_{\rm r}^2\lambda_{\max}^{\rm r}\mbbe\bem\|\mbf A(\mbf z^k-\mbf A^\dag\mbf b)\|^2\eem+(1+\ve)\eta_{\rm r}\mbbe\bem\|\mbf x^{k-1}-\mbf x_\star^0\|^2\eem\\
&\leq(1+1/\ve)\alpha_{\rm r}^2\lambda_{\max}^{\rm r}\eta_{\rm c}^k\|\mbf A(\mbf z^0-\mbf A^\dag\mbf b)\|^2+(1+\ve)\eta_{\rm r}\mbbe\bem\|\mbf x^{k-1}-\mbf x_\star^0\|^2\eem\\
&\leq (1+1/\ve)\alpha_{\rm r}^2\lambda_{\max}^{\rm r}\|\mbf A(\mbf z^0-\mbf A^\dag\mbf b)\|^2(\eta_{\rm c}^k+\eta_{\rm c}^{k-1}(1+\ve)\eta_{\rm r})\\
& \ \quad +(1+\ve)^2\eta_{\rm r}^2\mbbe\bem\|\mbf x^{k-2}-\mbf x_\star^0\|^2\eem\\
&\leq \cdots\\
&\leq (1+1/\ve)\alpha_{\rm r}^2\lambda_{\max}^{\rm r}\|\mbf A(\mbf z^0-\mbf A^\dag\mbf b)\|^2\sum_{i=0}^{k-1}\eta_{\rm c}^{k-i}(1+\ve)^i\eta_{\rm r}^i\\
&\ \quad +(1+\ve)^k\eta_{\rm r}^k\|\mbf x^0-\mbf x_\star^0\|^2.
\end{align*} This completes the proof.
\end{proof}

\begin{remark} In Theorem \ref{nrgs}, no assumptions about the dimensions or rank of $\mbf A$ are assumed, and the system $\bf Ax=b$ can be consistent or inconsistent. Let $\eta=\max\{\eta_{\rm r},\eta_{\rm c}\}$. It follows from $0<\alpha_{\rm c}<2/\lambda_{\max}^{\rm c}$ and  $0<\alpha_{\rm r}<2/\lambda_{\max}^{\rm r}$ that $\eta<1$. Assume that $\ve$ satisfies $(1+\ve)\eta<1$. We have $$\mbbe\bem\|\mbf x^k-\mbf x_\star^0\|^2\eem\leq(1+\ve)^k\eta^k(\|\mbf x^0-\mbf x_\star^0\|^2+(1+\ve)\alpha_{\rm r}^2\lambda_{\max}^{\rm r}\|\mbf A(\mbf z^0-\mbf A^\dag\mbf b)\|^2/\ve^2),$$ which shows that the {\rm EBCRSI} algorithm converges linearly  in the mean square sense to $\mbf x_\star^0$ with the rate $(1+\ve)\eta$. 
\end{remark}

\subsection{The randomized extended Gauss--Seidel algorithm} The REGS algorithm \cite{ma2015conve} is one special case of the EBCRSI algorithm. Choosing $\dsp\mbf S=\frac{\|\mbf A\|_\rmf}{\|\mbf A_{i,:}\|}\mbf I_{:,i}$ with probability $\dsp\frac{\|\mbf A_{i,:}\|^2}{\|\mbf A\|_\rmf^2}$ and $\dsp\mbf T=\frac{\|\mbf A\|_\rmf}{\|\mbf A_{:,j}\|}\mbf I_{:,j}$  with probability $\dsp\frac{\|\mbf A_{:,j}\|^2}{\|\mbf A\|_\rmf^2}$, we have $$\mbbe\bem\mbf S\mbf S^\top\eem=\mbf I,\quad \mbbe\bem\mbf T\mbf T^\top\eem=\mbf I,$$ and obtain 
\begin{align*}
\mbf z^k & =\mbf z^{k-1}-\alpha_{\rm c}\frac{\|\mbf A\|_\rmf^2}{\|\mbf A_{:,j}\|^2}(\mbf A_{:,j})^\top(\mbf A\mbf z^{k-1}-\mbf b)\mbf I_{:,j},\\
\mbf x^k & =\mbf x^{k-1}-\alpha_{\rm r}\frac{\|\mbf A\|_\rmf^2}{\|\mbf A_{i,:}\|^2}\mbf A_{i,:}(\mbf x^{k-1}-\mbf z^k)(\mbf A_{i,:})^\top.
	\end{align*} We have $\lambda_{\max}^{\rm r}=\lambda_{\max}^{\rm c}=\|\mbf A\|_\rmf^2$. Setting $\alpha_{\rm r}=\alpha_{\rm c}=1/\|\mbf A\|_\rmf^2$, we recover the REGS algorithm \cite{ma2015conve,du2019tight}.

\subsection{The extended block column and row uniform sampling algorithm} We propose one new special case of the EBCRSI algorithm by using uniform sampling and refer to it as the extended block column and row uniform sampling (EBCRUS) algorithm. Assume $1\leq\ell\leq \min\{m,n\}$. Let $\mcali$ (resp. $\mcalj$) denote the set consisting of the uniform sampling of $\ell$ different numbers of $[m]$  (resp. $[n]$). Setting $\dsp\mbf S=\sqrt{{m}/{\ell}}\mbf I_{:,\mcali}$ and $\mbf T=\sqrt{{n}/{\ell}}\mbf I_{:,\mcalj}$, we have $$\mbbe\bem\mbf S\mbf S^\top\eem=\frac{\dsp\frac{m}{\ell}}{\begin{pmatrix}m\\ \ell\end{pmatrix}}\sum_{\mcali\subseteq[m],\ |\mcali|=\ell}\mbf I_{:,\mcali}\mbf I_{:,\mcali}^\top=\mbf I_m,\quad \mbbe\bem\mbf T\mbf T^\top\eem=\frac{\dsp\frac{n}{\ell}}{\begin{pmatrix}n\\ \ell\end{pmatrix}}\sum_{\mcalj\subseteq[n],\ |\mcalj|=\ell}\mbf I_{:,\mcalj}\mbf I_{:,\mcalj}^\top=\mbf I_n,$$ and obtain 
\begin{align*}
\mbf z^k &=\mbf z^{k-1}-\alpha_{\rm c}\frac{n}{\ell}\mbf I_{:,\mcalj}(\mbf A_{:,\mcalj})^\top(\mbf A\mbf z^{k-1}-\mbf b),\\
\mbf x^k & =\mbf x^{k-1}-\alpha_{\rm r}\frac{m}{\ell}(\mbf A_{\mcali,:})^\top\mbf A_{\mcali,:}(\mbf x^{k-1}-\mbf z^k).
	\end{align*}  We have $$\lambda_{\max}^{\rm r}=\frac{m}{\ell}\max_{\mcali\subseteq[m],|\mcali|=\ell}\|\mbf A_{\mcali,:}\|^2, \quad\mbox{and}\quad \lambda_{\max}^{\rm c}=\frac{n}{\ell}\max_{\mcalj\subseteq[n], |\mcalj|=\ell}\|\mbf A_{:,\mcalj}\|^2.$$ By Theorem \ref{nrgs}, the EBCRUS algorithm can have a faster convergence rate than that of the REGS algorithm if there exists $1\leq\ell\leq\min\{m,n\}$ satisfying $$\frac{m}{\ell}\max_{\mcali\subseteq[m],|\mcali|=\ell}\|\mbf A_{\mcali,:}\|^2\leq\|\mbf A\|_\rmf^2,\quad\mbox{and}\quad \frac{n}{\ell}\max_{\mcalj\subseteq[n], |\mcalj|=\ell}\|\mbf A_{:,\mcalj}\|^2\leq\|\mbf A\|_\rmf^2.$$


\end{document}